\documentclass[11pt,reqno]{amsart}

\usepackage[margin=1.1in]{geometry}
\usepackage{amsmath,amssymb,amsthm,mathtools}
\usepackage{ytableau}
\usepackage{tikz}
\usepackage[colorlinks=true,linkcolor=blue,citecolor=blue]{hyperref}

\theoremstyle{plain}
\newtheorem{theorem}{Theorem}[section]
\newtheorem{proposition}[theorem]{Proposition}
\newtheorem{lemma}[theorem]{Lemma}
\newtheorem{corollary}[theorem]{Corollary}
\newtheorem{conjecture}[theorem]{Conjecture}

\theoremstyle{definition}
\newtheorem{definition}[theorem]{Definition}
\newtheorem{example}[theorem]{Example}
\newtheorem{remark}[theorem]{Remark}

\newcommand{\BK}{\mathrm{BK}}
\newcommand{\BKR}{\mathrm{BKR}}
\newcommand{\RC}{\mathcal{RC}}
\newcommand{\Cn}{\mathcal{C}}
\newcommand{\Gam}{\Gamma}
\newcommand{\SSYT}{\mathrm{SSYT}}
\newcommand{\ZZ}{\mathbb{Z}}

\newcommand{\Sym}{\mathfrak{S}}
\newcommand{\ev}{\mathsf{e}}
\newcommand{\rev}{\mathrm{rev}}
\newcommand{\id}{\mathrm{id}}

\DeclareMathOperator{\wt}{wt}

\begin{document}

\title{The big cactus group, invariant tensors, and rectangular tableaux}
\author{Gabriel Frieden}
\author{Pavlo Pylyavskyy}
\author{Leonid Rybnikov}
\date{\today}

\begin{abstract}
The cactus group acts on tensor products of crystals and, through Bender--Knuth involutions, on semistandard Young tableaux. We show that the big cactus group $J'_n=\pi_1^{\mathfrak S_n}(\overline M_{0,n}(\mathbb R))$, $n\ge3$, is obtained from the cactus group by imposing $n-1$ Dynkin-reversal relations, which say that conjugation by the full reversal acts on the Bender--Knuth generators by the Dynkin diagram automorphism; the quotient map is induced by forgetting the distinguished marked point. We give a second presentation, in which these relations are replaced by the requirement that every reduced word for the longest element of $\mathfrak S_n$ adapted to an orientation of $A_{n-1}$ represents the full reversal. These relations also hold on rectangular tableaux, where every adapted reduced word therefore acts as evacuation, and on invariant tensors: for finite normal crystals of any complex semisimple Lie algebra, the cactus action on invariants factors through $J'_n$. We prove the latter combinatorially in type $A$ and in general via Gaudin eigenline coverings. Haiman's promotion relation follows from a single adapted-word relation, but for $n\ge4$ it is strictly weaker than the whole family.
\end{abstract}

\maketitle

\setcounter{tocdepth}{1}
\tableofcontents

\section{Introduction}

\subsection{Bender--Knuth involutions and their group}

Bender--Knuth involutions $t_1,\dots,t_{n-1}$ act on semistandard Young
tableaux with entries in $\{1,\dots,n\}$. The involution $t_i$ rearranges
the entries $i$ and $i+1$ so as to interchange the number of $i$'s and
$i+1$'s, while preserving the shape and all other entries. They were
introduced to give a combinatorial proof that Schur functions are
symmetric.

The individual involutions $t_i$ are simple, but the group they generate
is not. Berenstein and Kirillov \cite{BK} initiated the study of the group
$\BK_n$ generated by $t_1,\dots,t_{n-1}$ acting on semistandard tableaux
of \emph{all} shapes, now called the Berenstein--Kirillov group. Two
distinguished elements are
\[
  \text{promotion}\quad p_{n-1}=t_{n-1}t_{n-2}\cdots t_1,
  \qquad
  \text{evacuation}\quad
  q_{n-1}=t_1(t_2t_1)(t_3t_2t_1)\cdots(t_{n-1}\cdots t_1).
\]
Finding a presentation of $\BK_n$ is still open, and appears to be hard.

Chmutov, Glick and the second author \cite{CGP} showed that the
Bender--Knuth involutions satisfy the relations of the cactus group
$\Cn_n$ of Henriques and Kamnitzer \cite{HK}, giving a surjection
$\Cn_n\twoheadrightarrow\BK_n$. The same conclusion was obtained
independently by Halacheva \cite{Halacheva} using skew Howe duality, and
concurrently by Berenstein and Kirillov \cite{BK16} by a group-theoretic
argument. Thus the cactus group gives a large family of relations in
$\BK_n$ that come from a general structure rather than from computation.

\subsection{Evacuation words on a rectangle}

This paper concerns what happens when one restricts the Bender--Knuth
action to \emph{rectangular} shapes, which in type $A$ model
Kirillov--Reshetikhin crystals. Let $\BKR_n$ be the quotient of $\BK_n$
obtained by imposing all relations that hold on semistandard tableaux of
every rectangular shape with entries in $\{1,\dots,n\}$.

Call a reduced word $\mathbf{i}=(i_1,\dots,i_N)$ for
$w_0\in\Sym_n$ an \emph{evacuation word} if the corresponding product
$t_{i_1}\cdots t_{i_N}$ acts as evacuation. The standard product
\[
q_{n-1}=t_1(t_2t_1)(t_3t_2t_1)\cdots(t_{n-1}\cdots t_1)
\]
is an evacuation word, as is every commutation-equivalent reduced word. On general shapes we know of no others, and conjecture that there
are none (Conjecture~\ref{conj:evac-BK}, verified for $n\le5$). On
rectangles, many more words work.

Recall that a reduced word for $w_0$ is \emph{adapted} to an orientation
$\Omega$ of the Dynkin diagram $A_{n-1}$ if $i_1$ is a source of
$\Omega$, $i_2$ is a source of $s_{i_1}\Omega$, and so on. Each reduced
word is adapted to at most one orientation, and the words adapted to a
fixed $\Omega$ form a single commutation class
\cite{Lusztig,Bedard}.

\begin{theorem}[Theorem~\ref{thm:knit-are-evac}]\label{thm:A-intro}
On rectangular tableaux, every reduced word for $w_0$ adapted to an
orientation of $A_{n-1}$ is an evacuation word.
\end{theorem}

Thus the single commutation class coming from the linear orientation is replaced on rectangles by
all $2^{n-2}$ adapted classes. We conjecture that these are all the reduced words representing
evacuation on rectangles.

\begin{definition}
A reduced word $w$ for $w_0\in\Sym_n$ is an \emph{evacuation word} in $\BK_n$
(resp.\ $\BKR_n$) if $t_w=q_{n-1}$ in $\BK_n$ (resp.\ $\BKR_n$).
\end{definition}

\begin{conjecture}\label{conj:evac-BK}
The evacuation words in $\BK_n$ are exactly the reduced words for $w_0$ adapted to the linear
orientation. 
\end{conjecture}

Here the linear orientation is
\[
1\longrightarrow2\longrightarrow\cdots\longrightarrow n-1.
\]
For the opposite linear orientation, the adapted words represent
$\alpha(q_{n-1})$, the dual evacuation, on arbitrary tableaux. On rectangles,
dual evacuation agrees with evacuation; equivalently, $\alpha(q_{n-1})=q_{n-1}$.
Thus Theorem~\ref{thm:A-intro} extends this identity from the opposite linear
orientation to every orientation of $A_{n-1}$.

\begin{conjecture}\label{conj:evac-BKR}
The evacuation words in $\BKR_n$ are exactly the reduced words for $w_0$ adapted to some
orientation of $A_{n-1}$.
\end{conjecture}

We verified both conjectures by computer for $n\le5$.

\subsection{Equivalent relations}

The exponentially large family of evacuation identities from
Theorem~\ref{thm:A-intro} is equivalent to a very small set of
relations. Let $\alpha$ be the index-reversing map
$\alpha(t_i)=t_{n-i}$. The \emph{Dynkin-reversal relations}
are
\begin{equation}\tag{$\ast$}\label{eq:star-intro}
  q_{n-1}\,t_i\,q_{n-1}=t_{n-i},
  \qquad i=1,\dots,n-1.
\end{equation}
In other words, conjugation by evacuation acts on the Bender--Knuth
generators by the Dynkin diagram automorphism.

\begin{theorem}[Theorems~\ref{thm:first-kind} and~\ref{thm:equivalence}]
\label{thm:B-intro}
The relations \eqref{eq:star-intro} hold in $\BKR_n$. Moreover, modulo
the cactus relations, they are equivalent to the whole family of
Theorem~\ref{thm:A-intro}.
\end{theorem}

\subsection{Promotion}

We also obtain the identity
\[
  q_{n-1}\,\alpha(q_{n-1})
  =(t_1t_2\cdots t_{n-1})^n
  =p_{n-1}^{-n},
\]
which already holds in the commutation group
(Theorem~\ref{thm:promotion-identity}). For the opposite linear orientation
\[
1\leftarrow2\leftarrow\cdots\leftarrow n-1,
\]
the corresponding adapted relation on rectangles is
\[
\alpha(q_{n-1})=q_{n-1}.
\]
Together with $q_{n-1}^2=1$, this gives $p_{n-1}^n=1$ on rectangles,
which is Haiman's theorem \cite[Thm.~4.4]{Haiman}.	
We  show that for every $n\ge4$, the relation
$p_{n-1}^n=1$ is strictly weaker than the relations
\eqref{eq:star-intro}
(Proposition~\ref{prop:promotion-weaker}). Thus Haiman's relation captures only part of the relations we describe on rectangles.

\subsection{Adapted words and sorting networks}

A reduced word for $w_0$ can be viewed as a sorting network on $n$
wires. Choosing four wires and keeping only the crossings between them
gives a reduced word for $w_0\in\Sym_4$, which we call a
\emph{$4$-subnetwork}. Warrington \cite{Warrington} introduced this
point of view in connection with Sylvester's four-point problem. Under
the Goodman--Pollack correspondence with planar point configurations,
the sixteen reduced words for $w_0\in\Sym_4$ split into twelve
corresponding to four points in convex position and four,
\[
  X=\{123212,\ 321232,\ 212321,\ 232123\},
\]
corresponding to configurations in which one point lies inside the
triangle formed by the other three.

Meng \cite{Meng} studied reduced words avoiding these four subnetworks
and asked how many $X$-avoiding reduced words for $w_0\in\Sym_n$ there
are. The counts reported by Warrington begin
\[
2,12,328,54520,68641152.
\]
These are exactly the numbers of adapted reduced words that we obtain.
In fact, we prove the following.

\begin{theorem}[Theorem~\ref{thm:X-avoiding}]
Every adapted reduced word for $w_0$ is $X$-avoiding.
\end{theorem}

\subsection{The rectangular cactus group}

Let $\RC_n$, the \emph{rectangular cactus group}, be the quotient of
$\Cn_n$ by the Dynkin-reversal relations. For $n\geq3$,
Theorem~\ref{thm:big-cactus} identifies this quotient with the big
cactus group $J'_n$ via the map that forgets the root marked point.
By construction,
\[
\Cn_n\twoheadrightarrow \RC_n\twoheadrightarrow \BKR_n.
\]
The second map need not be an isomorphism: already for $n=4$,
$
\RC_4\twoheadrightarrow\BKR_4
$
has nontrivial kernel.

\begin{theorem}[Theorem~\ref{thm:rc4-not-bkr4}]
The natural map
$
\RC_4\twoheadrightarrow\BKR_4
$
is not an isomorphism.
\end{theorem}

\begin{conjecture}[Conjecture~\ref{conj:relative}]
\[
\BKR_n
\cong
\BK_n\big/
\big\langle\!\big\langle
q_{n-1}t_iq_{n-1}t_{n-i}
\big\rangle\!\big\rangle .
\]
\end{conjecture}

In other words, we conjecture that the Dynkin-reversal relations generate
all additional relations that appear when one restricts from arbitrary
shapes to rectangles. This avoids the much harder problem of finding a
presentation of $\BK_n$ itself.

\subsection{Crystals and invariant tensors}

The cactus group of Henriques and Kamnitzer \cite{HK} acts on tensor products
$B_1\otimes\cdots\otimes B_n$ of crystals: the generator $q_{[i,j]}$ reverses the factors
$i,\dots,j$ using the crystal commutor, built from the Sch\"utzenberger involution. We show that
the group $\RC_n\cong J'_n$ has a natural meaning in this setting.

An \emph{invariant} of a crystal is an element on which every $e_i$ and every $f_i$ is undefined,
equivalently a one-element connected component. We write $\operatorname{Inv}(B)$ for the set of
invariants. Since the cactus action is by crystal isomorphisms, it carries invariants to invariants
of the tensor product with the factors permuted.

\begin{theorem}\label{thm:inv-adapted}
Let $B_1,\dots,B_n$ be finite normal crystals for $\mathfrak{sl}_m$. On
$\operatorname{Inv}(B_1\otimes\cdots\otimes B_n)$, every reduced word for $w_0$ adapted to an
orientation of $A_{n-1}$, read in the elements $t_i\in\Cn_n$ of
Theorem~\ref{thm:CGP}, acts by
\[
b_1\otimes\cdots\otimes b_n
\longmapsto
\xi(b_n)\otimes\cdots\otimes\xi(b_1),
\]
where $\xi$ is the Sch\"utzenberger involution. Equivalently, the cactus action on invariants
factors through $\RC_n$.
\end{theorem}

We give a combinatorial proof in type $A$ and, in
Section~\ref{subsec:inv-topological}, a topological proof valid for
finite normal crystals of any complex semisimple Lie algebra; see
Corollary~\ref{cor:inv-adapted-all-types}.

Theorem~\ref{thm:knit-are-evac} is the special case in which all factors are single-row crystals;
the standard-tableau case is obtained when all factors are $B(\varpi_1)$. Thus the tableau
relations studied above are the simplest instances of a more general phenomenon for invariant
tensors.

\subsection{The big cactus group and the geometry of invariant tensors}

The rectangular cactus group has a topological interpretation.
The ordinary cactus group is the equivariant fundamental group
\cite{DJS,HK}
\[
 \Cn_n=\pi_1^{\Sym_n}\bigl(\overline M_{0,n+1}(\mathbb R)\bigr),
\]
where $\Sym_n$ permutes $n$ marked points and fixes the remaining
point, denoted $\infty$. Henriques's \emph{big cactus group}
\cite{HenriquesBig} is
\[
 J'_n=\pi_1^{\Sym_n}\bigl(\overline M_{0,n}(\mathbb R)\bigr).
\]
Thus the distinction is whether the stable curve has a distinguished
root in addition to the $n$ permuted markings.

\begin{theorem}[Theorem~\ref{thm:big-cactus}]
\label{thm:big-cactus-intro}
For $n\geq3$, forgetting $\infty$ induces an isomorphism
\[
 \Cn_n\Big/\big\langle\!\big\langle
 qt_iq\,t_{n-i}:1\leq i<n
 \big\rangle\!\big\rangle
 \;\xrightarrow{\sim}\;J'_n.
\]
In particular, $\RC_n\cong J'_n$.
\end{theorem}

Consequently, the Dynkin-reversal relations and the adapted-word
relations give two presentations of the big cactus group in the
$t_i$ generators. A different geometric presentation was given by
Davis--Januszkiewicz--Scott \cite[Section~8.6]{DJS}; our description
identifies the kernel of the root-forgetting map explicitly.
The proof uses the caterpillar paths and the mirabolic generation
theorem of Gurenkova, Halacheva and the third author~\cite{GHR}.

This interpretation also explains the relations on invariant tensors.
By the Gaudin monodromy theorem~\cite{HKRW}, the cactus action on
tensor-product multiplicity sets is realized as monodromy of a
covering of $\overline M_{0,n+1}(\mathbb R)$. For the multiplicity
of the trivial representation, this covering is pulled back from
$\overline M_{0,n}(\mathbb R)$. On the smooth locus, the descent
follows from the projective invariance of the reduced Gaudin algebra
\cite[Theorem~4.15(2)]{IKR}; in
Section~\ref{subsec:inv-topological}, we extend it to the boundary.
This gives a topological proof of Theorem~\ref{thm:inv-dynkin} and
extends Theorem~\ref{thm:inv-adapted} to finite normal crystals of
arbitrary semisimple type.

The proof of Theorem~\ref{thm:big-cactus} uses only Sections~\ref{sec:prelim}
and~\ref{sec:first}, so Section~\ref{sec:big-cactus} can be read directly after them.

\subsection{Related work}

The big cactus group $J'n=\pi_1^{\Sym_n}(\overline M{0,n}(\mathbb R))$ was introduced by Henriques~\cite{HenriquesBig}.
Our presentation of it in Bender--Knuth generators uses two ingredients from
forthcoming work of Gurenkova, Halacheva and the third author~\cite{GHR}: a
geometric realization of the elements $t_i$ as paths in the real
Deligne--Mumford space, and a generation theorem for the mirabolic cactus
group.

The connection between Bender--Knuth involutions and the ordinary cactus
group was established independently by Chmutov, Glick and the second
author~\cite{CGP}, by Halacheva~\cite{Halacheva}, and by Berenstein and
Kirillov~\cite{BK16}. The closest precursor of
Theorem~\ref{thm:knit-are-evac} is the work of Garver, Patrias and
Thomas~\cite[Thm.~1.10]{GPT}. They show that reflection at a source of a
Dynkin quiver corresponds to a piecewise-linear toggle on the associated
minuscule reverse plane partition. Our knitting argument uses the same basic
idea in a more elementary form, following source reflections along a full
adapted reduced word for $w_0$.

The local-rule formulation used in Section~\ref{sec:invariants} is due to
Westbury~\cite[Def.~3.1]{Westbury18}. The general proof for invariant tensors
also uses the realization of the cactus action as monodromy of Gaudin
eigenlines and the projective invariance of the reduced Gaudin algebra
\cite{HKRW,IKR}. Further connections, including Bender--Knuth groups in
other settings, sorting networks, and related cactus actions, are discussed
in Section~\ref{sec:related}.

\subsection*{AI statement}
Claude and ChatGPT helped discover Theorem~\ref{thm:inv-dynkin} and its combinatorial proof, and suggested, with
computational evidence, that the quotient of $\Cn_n$ by the Dynkin-reversal relations is the big
cactus group. AI was also used extensively in editing.

\section{Preliminaries}\label{sec:prelim}

\subsection{Tableaux and Bender--Knuth involutions}

We identify a partition $\lambda=(\lambda_1\ge\lambda_2\ge\cdots)$ with its Young diagram.
A \emph{semistandard Young tableau} (SSYT) of shape $\lambda$ with entries in
$[n]=\{1,\dots,n\}$ is a filling of $\lambda$ that is weakly increasing along rows and strictly
increasing down columns. We write $\SSYT(\lambda,n)$ for the set of such tableaux, and
\[
\wt(T)=(\wt_1(T),\dots,\wt_n(T))
\]
for the content of $T$, where $\wt_a(T)$ is the number of entries equal to $a$.

Fix $i\in\{1,\dots,n-1\}$. An entry $i$ is called \emph{free} if there is no $i+1$ directly below it,
and an entry $i+1$ is called free if there is no $i$ directly above it. In each row, the free $i$'s
and $(i+1)$'s form a block
\[
i^a(i+1)^b
\]
for some $a,b\ge 0$. The \emph{Bender--Knuth involution} $t_i$ replaces this block by
\[
i^b(i+1)^a
\]
in every row. The resulting tableau is again semistandard. Thus $t_i$ exchanges the numbers of
$i$'s and $(i+1)$'s:
\[
\wt(t_iT)=s_i\wt(T),
\]
where $s_i$ interchanges the $i$-th and $(i+1)$-st coordinates. In particular, $t_i^2=\id$.

\begin{example}\label{ex:BK}
Let
\[
T=\begin{ytableau}
1&1&2\\
2&3
\end{ytableau}
\in \SSYT((3,2),3),
\qquad
\wt(T)=(2,2,1).
\]

For $i=2$, all three entries $2,2,3$ are free. The free block in the first row is $2$, so it becomes
$3$, while the block $23$ in the second row is unchanged. Hence
\[
t_2T=
\begin{ytableau}
1&1&3\\
2&3
\end{ytableau},
\qquad
\wt(t_2T)=(2,1,2).
\]

For $i=1$, the $1$ in the first column and the $2$ below it are not free. The remaining free block
in the first row is $12$, so it is unchanged. Thus
$
t_1T=T.
$
\end{example}

\begin{definition}
The \emph{Berenstein--Kirillov group} $\BK_n$ is the group generated by
$t_1,\dots,t_{n-1}$ acting on
\[
\bigsqcup_{\lambda}\SSYT(\lambda,n),
\]
where the union is over all partitions $\lambda$. The \emph{rectangular Berenstein--Kirillov
group} $\BKR_n$ is the group generated by the same operators acting only on rectangular shapes:
\[
\bigsqcup_{k,m\ge1}\SSYT((m^k),n).
\]
\end{definition}

Thus a word in the $t_i$ is trivial in $\BK_n$ if and only if it acts trivially on tableaux of
every shape, and trivial in $\BKR_n$ if and only if it acts trivially on every rectangle. Thus $\BKR_n$ is a quotient of $\BK_n$, via restriction to rectangular shapes:
\[
\BK_n\twoheadrightarrow \BKR_n.
\]

We will use repeatedly the following fact, due to Berenstein and Kirillov
\cite[Corollary~1.1]{BK}; see also \cite[Remark~1.9]{CGP}.

\begin{lemma}\label{lem:t1t2}
For every $n\ge3$,
\[
(t_1t_2)^6=1
\]
in $\BK_n$, and hence also in $\BKR_n$.
\end{lemma}

The analogous statement for other pairs of adjacent generators is false.

\subsection{Evacuation, promotion, and the elements \texorpdfstring{$q_{[i,j]}$}{q[i,j]}}

Following \cite{BK,CGP}, define
\[
p_i=t_it_{i-1}\cdots t_1,
\qquad
q_i=t_1(t_2t_1)(t_3t_2t_1)\cdots(t_it_{i-1}\cdots t_1)
     =q_{i-1}p_i,
\]
and, for $1\le i<j\le n$,
\[
q_{[i,j]}=q_{j-1}\,q_{j-i}\,q_{j-1}.
\]
The element $p_{n-1}$ acts as \emph{promotion}, while $q_{n-1}$ acts as
\emph{evacuation}. We write
$
p:=p_{n-1},
$
and
$
q:=q_{n-1}.
$
In particular, $q_{[1,j]}=q_{j-1}$.
We compose operators from the right: in
$t_{j_1}t_{j_2}\cdots t_{j_N},$
the factor $t_{j_N}$ acts first.

The following classical fact identifies the elements $q_i$ with evacuation
\cite[Theorem~2.1]{BK}; see also \cite[Proposition~4.2]{CGP}.

\begin{lemma}\label{lem:q-is-evac}
For every partition $\lambda$, the element $q_{n-1}$ acts on
$\SSYT(\lambda,n)$ as Sch\"utzenberger evacuation $\ev_\lambda$.
\end{lemma}

We will use the following standard description of evacuation. Let $w(T)$ be the row reading word of
$T$, obtained by reading the rows from left to right, starting with the bottom row. If $P(w)$
denotes the RSK row-insertion tableau, then
\begin{equation}\label{eq:rc}
  \ev_\lambda(T)=P\bigl(\mathrm{rc}(w(T))\bigr),
\end{equation}
where $\mathrm{rc}$ reverses the word and replaces each letter $a$ by $n+1-a$;
see \cite[App.~A1]{Stanley} and \cite{Fulton}.

Evacuation is an involution and reverses the content:
\[
\ev_\lambda^2=\id,
\qquad
\wt(\ev_\lambda T)=w_0\cdot\wt(T),
\]
where $w_0$ reverses the coordinates; see \cite[App.~A]{Fulton} and
\cite[App.~A1]{Stanley}.

\subsection{The cactus group}

The \emph{cactus group} $\Cn_n$ \cite{HK} is generated by elements
$q_{[i,j]}$, $1\le i<j\le n$, with relations
\begin{enumerate}
\item[(C1)] $q_{[i,j]}^2=1$;
\item[(C2)] $q_{[i,j]}q_{[k,l]}=q_{[k,l]}q_{[i,j]}$ if $j<k$;
\item[(C3)] $q_{[i,j]}q_{[k,l]}q_{[i,j]}
 =q_{[i+j-l,\,i+j-k]}$ if $i\le k<l\le j$.
\end{enumerate}

We use the same notation $q_{[i,j]}$ for the cactus generators and for their images in
$\BK_n$; the ambient group will always be clear from context.
The connection with the Berenstein--Kirillov group is given by the following result.

\begin{theorem}[{\cite[Thms.~1.4, 1.7, 1.8]{CGP}}]\label{thm:CGP}
The Bender--Knuth operators $q_{[i,j]}$ satisfy the cactus relations, giving a natural
quotient
\[
\Cn_n\twoheadrightarrow\BK_n.
\]

The cactus group also has a presentation with generators $t_1,\dots,t_{n-1}$. With
\[
q_i=t_1(t_2t_1)\cdots(t_it_{i-1}\cdots t_1),
\qquad
q_{[i,j]}=q_{j-1}q_{j-i}q_{j-1},
\]
the defining relations are
\[
t_i^2=1,\qquad
t_it_j=t_jt_i \quad (|i-j|>1),\qquad
(t_iq_{[j,k]})^2=1 \quad (i+1<j<k).
\]
Conversely,
$
t_1=q_{[1,2]},\qquad
t_2=q_{[1,2]}q_{[1,3]}q_{[1,2]},
$
and, for $i\ge3$,
$
t_i=q_{[1,i]}q_{[1,i+1]}q_{[1,i]}q_{[1,i-1]}.
$
\end{theorem}

We will often work first in the following group, where only the ``obvious'' relations are imposed.

\begin{definition}
The \emph{commutation group} is the right-angled Coxeter group
\[
\Gam_n=
\big\langle t_1,\dots,t_{n-1}\mid
t_i^2=1,\;
t_it_j=t_jt_i\ \text{for }|i-j|>1
\big\rangle .
\]
\end{definition}

Any identity in $\Gam_n$ therefore holds in $\Cn_n$, $\BK_n$, and $\BKR_n$.

The index-reversal map
$
\alpha:t_i\longmapsto t_{n-i}
$
preserves the defining relations of $\Gam_n$ and hence is an automorphism of $\Gam_n$.
It does not, in general, descend to the cactus group: already for $n=4$, the image under
$\alpha$ of the relation $(t_1q_{[3,4]})^2=1$ fails in $\Cn_4$.
This is why we state the results below at the level of words in $\Gam_n$.

For a word $w=(j_1,\dots,j_N)$, write
$
t_w=t_{j_1}\cdots t_{j_N},
\rev(w)=(j_N,\dots,j_1).
$
Then $t_{\rev(w)}=t_w^{-1}$.

\section{Dynkin-reversal relations}\label{sec:first}

Throughout this section, $\lambda=(m^k)$ is a $k\times m$ rectangle.

\subsection{Evacuation on a rectangle}

\begin{definition}
For $T\in\SSYT((m^k),n)$, let $\theta(T)$ be obtained by rotating $T$ by $180^\circ$ and
replacing each entry $a$ by $n+1-a$. Thus
$
\theta(T)_{r,c}=n+1-T_{k+1-r,\,m+1-c}.
$
\end{definition}

Both rotation and complementation reverse order, so their composition preserves weak increase
along rows and strict increase down columns. Hence
\[
\theta(T)\in\SSYT((m^k),n),
\qquad
\theta^2=\id.
\]

\begin{lemma}\label{lem:theta-is-evac}
For rectangular $\lambda$, we have
\[
\theta=\ev_\lambda=q_{n-1}.
\]
\end{lemma}

This is standard. Berenstein and Zelevinsky \cite{BZ} describe evacuation on an arbitrary shape
as rotation, complementation, and rectification. For a rectangle, rotation again gives a straight
shape, so no rectification is needed. The same statement is explicitly made for row-strict
tableaux in \cite[Lem.~3.1]{PP}. 

\begin{lemma}\label{lem:theta-conj}
On rectangular shapes,
\[
\theta\, t_i\, \theta=t_{n-i},
\qquad 1\le i\le n-1.
\]
\end{lemma}

\begin{proof}
Write $\theta=\kappa\circ\rho$, where $\rho$ is rotation by $180^\circ$ and
$\kappa(a)=n+1-a$. Thus $\{i,i+1\}$ is sent to
$\{n-i,n-i+1\}$, with the two values interchanged.

Rotation also interchanges ``above'' and ``below''. Hence the non-free pairs for $t_i$
are carried exactly to the non-free pairs for $t_{n-i}$, and therefore free entries are
carried to free entries.

If the free block in a row of $T$ is
$
i^a(i+1)^b,
$
then the corresponding free block in $\theta(T)$ is
$
(n-i)^b(n-i+1)^a.
$
Applying $t_{n-i}$ gives
$
(n-i)^a(n-i+1)^b,
$
which is the image under $\theta$ of the block
$
i^b(i+1)^a.
$
Thus $\theta t_i\theta=t_{n-i}$.
\end{proof}
\begin{theorem}\label{thm:first-kind}
In $\BKR_n$,
\[
q_{n-1}t_iq_{n-1}=t_{n-i},
\qquad 1\le i\le n-1.
\]
\end{theorem}

\begin{proof}
This follows immediately from Lemmas~\ref{lem:theta-is-evac}
and~\ref{lem:theta-conj}.
\end{proof}

\begin{example}
Let $n=4$, $\lambda=(2,2)$, and
\[
T=\begin{ytableau}1&1\\2&4\end{ytableau}.
\]
Then
\[
T \xmapsto{\,q_3\,}
\begin{ytableau}1&3\\4&4\end{ytableau}
\xmapsto{\,t_1\,}
\begin{ytableau}2&3\\4&4\end{ytableau}
\xmapsto{\,q_3\,}
\begin{ytableau}1&1\\2&3\end{ytableau}
=t_3T,
\]
illustrating $q_3t_1q_3=t_3$.
\end{example}

\subsection{Equivalent forms}

Define
\[
q^0_{[i,j]}
=
\alpha\bigl(q_{[n+1-j,\,n+1-i]}\bigr),
\]
where $\alpha$ acts on words by $t_a\mapsto t_{n-a}$.

\begin{example}
For $n=4$,
\[
\begin{aligned}
q^0_{[3,4]}
  &=\alpha(q_{[1,2]})
   =\alpha(t_1)
   =t_3,\\
q^0_{[2,4]}
  &=\alpha(q_{[1,3]})
   =\alpha(t_1t_2t_1)
   =t_3t_2t_3,\\
q^0_{[1,4]}
  &=\alpha(q_{[1,4]})
   =\alpha(t_1t_2t_1t_3t_2t_1)
   =t_3t_2t_3t_1t_2t_3.
\end{aligned}
\]
\end{example}

\begin{proposition}\label{prop:first-kind-equiv}
In any quotient of $\Cn_n$, the following are equivalent:
\begin{enumerate}
\item[(a)] $q_{[i,j]}=q^0_{[i,j]}$ for all $1\le i<j\le n$;
\item[(b)] $q_{[j,n]}=q^0_{[j,n]}$ for all $1\le j<n$;
\item[(c)] $q\,t_i\,q=t_{n-i}$ for all $1\le i\le n-1$.
\end{enumerate}
\end{proposition}
\begin{proof}
The implication (a)$\Rightarrow$(b) is immediate.

For (b)$\Rightarrow$(c), we have
$q_{[j,n]}=q\,q_{n-j}\,q$ and
$q^0_{[j,n]}=\alpha(q_{n-j})$, so (b) is equivalent to
\begin{equation}\label{eq:qk}
q\,q_k\,q=\alpha(q_k),\qquad 1\le k\le n-1.
\end{equation}
Since $q_k=q_{k-1}p_k$, comparing \eqref{eq:qk} for $k$ and $k-1$ gives
$q\,p_k\,q=\alpha(p_k)$. As $p_k=t_k\cdots t_1$, this becomes
\[
q\,t_k\cdots t_1\,q=t_{n-k}\cdots t_{n-1}.
\]
Comparing with the same identity for $k-1$ gives
$q\,t_k\,q=t_{n-k}$.

For (c)$\Rightarrow$(a), conjugation by $q$ acts on every word by $\alpha$. Hence
\[
q^0_{[i,j]}
=\alpha\bigl(q_{[n+1-j,n+1-i]}\bigr)
=q\,q_{[n+1-j,n+1-i]}\,q
=q_{[i,j]},
\]
where the last equality is the cactus relation (C3), using $q=q_{[1,n]}$.
\end{proof}
\begin{corollary}\label{cor:thm22}
In $\BKR_n$,
\[
q_{[i,j]}=q^0_{[i,j]}
\]
for all $1\le i<j\le n$.
\end{corollary}

\begin{definition}\label{def:RC}
The \emph{rectangular cactus group} is the quotient
\[
\RC_n
=
\Cn_n\Big/
\Big\langle\!\Big\langle
q_{[1,n]}t_iq_{[1,n]}t_{n-i},
\ 1\le i\le n-1
\Big\rangle\!\Big\rangle.
\]
Equivalently, $\RC_n$ is the quotient of $\Cn_n$ by the relations
$q_{[i,j]}=q^0_{[i,j]}$.
\end{definition}

By Corollary~\ref{cor:thm22}, the natural map to the rectangular Berenstein--Kirillov group
factors as
\[
\Cn_n\twoheadrightarrow\RC_n\twoheadrightarrow\BKR_n.
\]

\subsection{The groups \texorpdfstring{$\RC_n$}{RCn} and \texorpdfstring{$\BKR_n$}{BKRn}}
\label{sec:groups}

By construction, there is a natural surjection
\[
\RC_n\twoheadrightarrow\BKR_n.
\]
One might hope that it is an isomorphism. Already for $n=4$, this is false.

\begin{theorem}\label{thm:rc4-not-bkr4}
The map
$
\RC_4\twoheadrightarrow\BKR_4
$
is not an isomorphism.
\end{theorem}

\begin{proof}
Let
$
D_\infty=\langle x,y\mid x^2=y^2=1\rangle
$
be the infinite dihedral group. Define
\[
t_1,t_3\longmapsto x,
\qquad
t_2\longmapsto y.
\]
We claim that this defines a homomorphism $\RC_4\to D_\infty$.

The involution and commutation relations are immediate. Also
$
q_3=t_1t_2t_1t_3t_2t_1
$
maps to
$
xyxxyx=1.
$
Therefore the Dynkin-reversal relations $q_3t_iq_3=t_{4-i}$ are satisfied, since $t_1$ and $t_3$
have the same image. The only remaining cactus relation is
$(t_1q_{[3,4]})^2=1$. Since
$
q_{[3,4]}=q_3q_1q_3
$
maps to $x$, this relation is also satisfied.

Thus $t_1t_2$ maps to $xy$, which has infinite order in $D_\infty$, so $t_1t_2$ has infinite
order in $\RC_4$. On the other hand,
$
(t_1t_2)^6=1
$
in $\BKR_4$ by Lemma~\ref{lem:t1t2}. Therefore
$\RC_4\twoheadrightarrow\BKR_4$ has nontrivial kernel.
\end{proof}

Note that this failure comes from relations that already hold in $\BK_n$ but not in the
cactus group. This suggests asking instead whether the Dynkin-reversal relations account for all
relations introduced by restricting from arbitrary shapes to rectangles.

\begin{conjecture}\label{conj:relative}
The rectangular Berenstein--Kirillov group is obtained from $\BK_n$ by imposing the
Dynkin-reversal relations:
\[
\BKR_n
\cong
\BK_n\Big/
\Big\langle\!\Big\langle
q_{n-1}t_iq_{n-1}t_{n-i}
:\ 1\le i\le n-1
\Big\rangle\!\Big\rangle .
\]
\end{conjecture}

In Section~\ref{sec:big-cactus} we identify \(\RC_n\) with Henriques's big cactus group \(J'_n\) (Theorem~\ref{thm:big-cactus}). Conjecture~\ref{conj:relative} then says that \(\BKR_n\) is obtained from \(\BK_n\) by imposing the same additional relations that turn \(\Cn_n\) into \(J'_n\): restricting from arbitrary tableaux to rectangles should introduce exactly the relations introduced by forgetting the root.

\section{Knitting words}\label{sec:knitting}

\subsection{The strip, Dynkin paths, and flips}

Let
\[
\mathcal S=\{(i,j)\in\ZZ^2:1\le j\le n-1,\ i+j\text{ even}\}.
\]
A \emph{Dynkin path} is a sequence
\[
D=((d_1,1),\dots,(d_{n-1},n-1))
\]
with $|d_{j+1}-d_j|=1$ for all $j$. We usually identify $D$ with its height vector
$(d_1,\dots,d_{n-1})$.

A Dynkin path determines an orientation $\Omega_D$ of the Dynkin diagram $A_{n-1}$:
orient
\[
j\longrightarrow j+1
\quad\text{if } d_{j+1}=d_j+1,
\]
and $j+1\to j$ otherwise. Every orientation arises in this way, uniquely up to horizontal
translation of the path.

\begin{figure}[ht]
\centering
\begin{tikzpicture}[scale=0.72,>=stealth]
\foreach \j in {1,...,4} {
  \foreach \i in {-1,...,7} {
    \pgfmathparse{int(mod(\i+\j+100,2))}
    \ifnum\pgfmathresult=0 \fill[gray!35] (\i,\j) circle (2.2pt); \fi
  }
}
\foreach \j in {1,...,4} {\node at (-2.4,\j) {$j=\j$};}
\foreach \i in {-1,1,3,5,7} {\node[gray] at (\i,0.35) {\scriptsize $\i$};}
\node[gray] at (8.2,0.35) {\scriptsize $i$};
\draw[very thick,blue] (3,1)--(2,2)--(3,3)--(4,4);
\foreach \p/\j/\d in {3/1/3, 2/2/2, 3/3/3, 4/4/4} {
  \fill[blue] (\p,\j) circle (3pt);
  \node[blue,anchor=east] at (\p-0.25,\j) {\small $d_{\j}=\d$};
}
\node[blue] at (3,4.7) {$D$};
\begin{scope}[xshift=10.6cm]
  \foreach \j in {1,...,4} {\fill[black] (0,\j) circle (3.2pt); \node[anchor=west] at (0.25,\j) {$\j$};}
  \draw[very thick,->] (0,1.85)--(0,1.15);
  \draw[very thick,->] (0,2.15)--(0,2.85);
  \draw[very thick,->] (0,3.15)--(0,3.85);
  \node at (0,4.7) {$\Omega_D$};
\end{scope}
\end{tikzpicture}
\caption{The strip $\mathcal S$ for $n=5$ and the Dynkin path
$D=(3,2,3,4)$ (blue). The corresponding orientation of $A_4$ is
\[
\Omega_D:\qquad 1\leftarrow2\to3\to4.
\]
Thus $2$ is a source and $1,4$ are sinks.}
\label{fig:dynkin-path}
\end{figure}
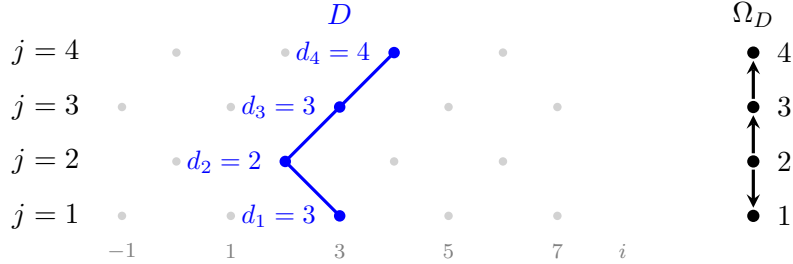

A vertex $j$ is a \emph{source} of $D$ if
$
d_{j'}=d_j+1
$
for every neighbour $j'\in\{j-1,j+1\}\cap[1,n-1]$. Equivalently, $j$ is a source of
$\Omega_D$.
The \emph{flip} at a source $j$ replaces $d_j$ by $d_j+2$ and leaves all other heights unchanged.
On orientations, this is reflection at the source $j$.

In coordinates, the flip is the local move
\[
(i,j),(i-1,j+1),(i,j+2)
\longmapsto
(i,j),(i+1,j+1),(i,j+2),
\]
with the obvious boundary versions at $j=1$ and $j=n-1$.

A \emph{knitting word} from $D$ to $D'$ is a sequence
$w=(j_1,\dots,j_N)$ such that successive flips at the sources
$j_1,\dots,j_N$ carry $D$ to $D'$. We write $W(D,D')$ for the set of all knitting words from
$D$ to $D'$.

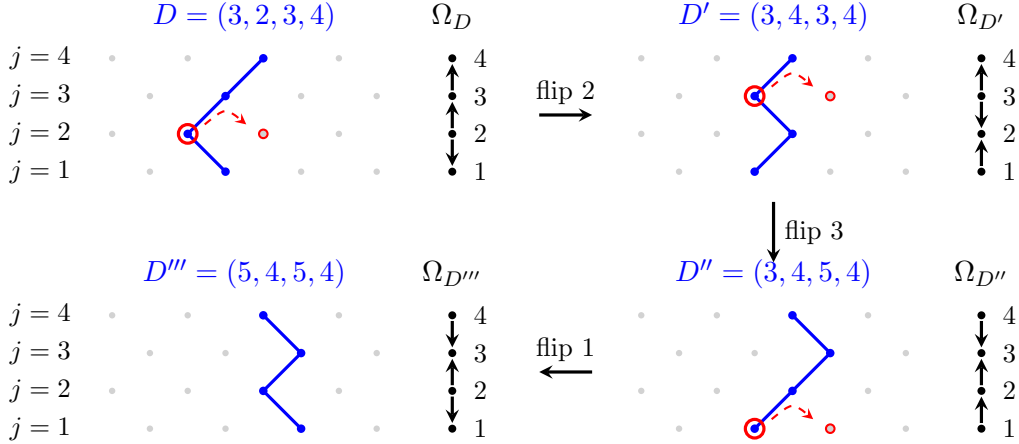
\begin{figure}[ht]
\centering
\begin{tikzpicture}[scale=0.5,>=stealth]
\newcommand{\strip}{%
  \foreach \j in {1,...,4} {\foreach \i in {0,...,7} {
    \pgfmathparse{int(mod(\i+\j+100,2))}
    \ifnum\pgfmathresult=0 \fill[gray!35] (\i,\j) circle (2.4pt); \fi }}}
\newcommand{\levels}{\foreach \j in {1,...,4} {\node at (-1.9,\j) {\small $j=\j$};}}
\newcommand{\dynkin}{\foreach \j in {1,...,4} {\fill (9,\j) circle (3pt); \node[anchor=west] at (9.3,\j) {\small $\j$};}}
\newcommand{\up}[1]{\draw[very thick,->] (9,#1+0.15)--(9,#1+0.85);}     
\newcommand{\down}[1]{\draw[very thick,->] (9,#1+0.85)--(9,#1+0.15);}   

\begin{scope}
  \strip \levels \dynkin
  \draw[very thick,blue] (3,1)--(2,2)--(3,3)--(4,4);
  \foreach \p/\j in {3/1,2/2,3/3,4/4} {\fill[blue] (\p,\j) circle (3.2pt);}
  \draw[red,very thick] (2,2) circle (7pt);                 
  \draw[red,thick,dashed,->] (2.45,2.25) .. controls (3,2.7) .. (3.55,2.25);
  \draw[red,thick] (4,2) circle (3.2pt);                    
  \node[blue] at (3.5,5.1) {$D=(3,2,3,4)$};
  \up{2} \up{3} \down{1}
  \node at (9,5.1) {$\Omega_D$};
\end{scope}
\draw[very thick,->] (11.3,2.5)--(12.7,2.5) node[midway,above] {\small flip $2$};

\begin{scope}[xshift=14cm]
  \strip \dynkin
  \draw[very thick,blue] (3,1)--(4,2)--(3,3)--(4,4);
  \foreach \p/\j in {3/1,4/2,3/3,4/4} {\fill[blue] (\p,\j) circle (3.2pt);}
  \draw[red,very thick] (3,3) circle (7pt);
  \draw[red,thick,dashed,->] (3.45,3.25) .. controls (4,3.7) .. (4.55,3.25);
  \draw[red,thick] (5,3) circle (3.2pt);
  \node[blue] at (3.5,5.1) {$D'=(3,4,3,4)$};
  \up{1} \down{2} \up{3}
  \node at (9,5.1) {$\Omega_{D'}$};
\end{scope}
\draw[very thick,->] (17.5,0.2)--(17.5,-1.4) node[midway,right] {\small flip $3$};

\begin{scope}[xshift=14cm,yshift=-6.8cm]
  \strip \dynkin
  \draw[very thick,blue] (3,1)--(4,2)--(5,3)--(4,4);
  \foreach \p/\j in {3/1,4/2,5/3,4/4} {\fill[blue] (\p,\j) circle (3.2pt);}
  \draw[red,very thick] (3,1) circle (7pt);
  \draw[red,thick,dashed,->] (3.45,1.25) .. controls (4,1.7) .. (4.55,1.25);
  \draw[red,thick] (5,1) circle (3.2pt);
  \node[blue] at (3.5,5.1) {$D''=(3,4,5,4)$};
  \up{1} \up{2} \down{3}
  \node at (9,5.1) {$\Omega_{D''}$};
\end{scope}
\draw[very thick,->] (12.7,-4.3)--(11.3,-4.3) node[midway,above] {\small flip $1$};

\begin{scope}[yshift=-6.8cm]
  \strip \levels \dynkin
  \draw[very thick,blue] (5,1)--(4,2)--(5,3)--(4,4);
  \foreach \p/\j in {5/1,4/2,5/3,4/4} {\fill[blue] (\p,\j) circle (3.2pt);}
  \node[blue] at (3.5,5.1) {$D'''=(5,4,5,4)$};
  \down{1} \up{2} \down{3}
  \node at (9,5.1) {$\Omega_{D'''}$};
\end{scope}
\end{tikzpicture}
\caption{A knitting word $w=(2,3,1)\in W(D,D''')$ for $n=5$. In each panel the source to be
flipped is circled, and the corresponding orientation of $A_4$ is shown on the right. Each flip
raises one height by $2$ and reverses the arrows incident to that vertex.}
\label{fig:knitting}
\end{figure}

\begin{lemma}\label{lem:sources-exist}
Every Dynkin path has a source. Moreover, any Dynkin path can be transformed by flips into a path
with any prescribed orientation. In particular, every orientation can be reached from the linear
path
\[
D_{\mathrm{lin}}=(1,2,\dots,n-1).
\]
\end{lemma}

\begin{proof}
A vertex of minimal height is a source. To prove the second statement, encode the orientation by
$\varepsilon_j=d_{j+1}-d_j\in\{\pm1\}$. A flip at an interior source moves a $-1$ one step to the
right, while flips at the two ends create a $-1$ on the left or delete one on the right. Thus any
sign vector can be transformed into $(+1,\ldots,+1)$ by repeatedly moving its rightmost $-1$ to
the right and deleting it at the end. Conversely, starting from $(+1,\ldots,+1)$, we can obtain
any prescribed sign vector by creating $-1$'s at the left end and moving them to their desired
positions, placing the rightmost ones first. Thus
every orientation can be reached from every other one.
\end{proof}

\subsection{Knitting words define elements of \texorpdfstring{$\Gam_n$}{Gamma_n}}

We now show that a knitting word from $D$ to $D'$ determines an element of $\Gam_n$
independent of the choice of flips. Assume $n\ge3$.

Say that $D'$ lies weakly to the right of $D$, and write $D\le D'$, if
$d'_j\ge d_j$ for all $j$. Let
\[
R(D,D')=\{(i,j)\in\mathcal S:d_j\le i<d'_j\}.
\]
Order $R(D,D')$ by
\[
(i-1,j\pm1)\lessdot(i,j)
\]
whenever both nodes lie in the region.

\begin{lemma}\label{lem:knitting-words}
Let $D'$ lie weakly to the right of $D$.
\begin{enumerate}
\item[\textup{(a)}] Knitting sequences from $D$ to $D'$ are in bijection with the linear
extensions of $R(D,D')$, with the knitting word given by the levels of the nodes.
\item[\textup{(b)}] Every knitting word contains the letter $j$ exactly
$(d'_j-d_j)/2$ times. In particular, $|w|=|R(D,D')|$.
\item[\textup{(c)}] Any two knitting words are related by commutations of adjacent letters
$j,j'$ with $|j-j'|>1$. Hence they define the same element
\[
Q(D,D')\in\Gam_n,
\]
and therefore the same elements of $\Cn_n$, $\BK_n$, and $\BKR_n$.
\end{enumerate}
\end{lemma}

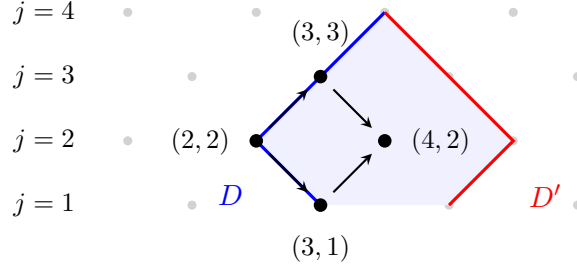
\begin{figure}[ht]
\centering
\begin{tikzpicture}[scale=0.85,>=stealth]
\foreach \j in {1,...,4} {\foreach \i in {0,...,7} {
  \pgfmathparse{int(mod(\i+\j+100,2))}
  \ifnum\pgfmathresult=0 \fill[gray!35] (\i,\j) circle (2pt); \fi }}
\foreach \j in {1,...,4} {\node at (-1.3,\j) {\small $j=\j$};}
\fill[blue!6] (3,1)--(2,2)--(3,3)--(4,4)--(5,3)--(6,2)--(5,1)--cycle;
\draw[very thick,blue] (3,1)--(2,2)--(3,3)--(4,4);
\draw[very thick,red]  (5,1)--(6,2)--(5,3)--(4,4);
\node[blue] at (1.6,1.15) {$D$};
\node[red]  at (6.5,1.15) {$D'$};
\draw[thick,->] (2.2,1.8)--(2.8,1.2);
\draw[thick,->] (2.2,2.2)--(2.8,2.8);
\draw[thick,->] (3.2,1.2)--(3.8,1.8);
\draw[thick,->] (3.2,2.8)--(3.8,2.2);
\foreach \p/\j in {2/2, 3/1, 3/3, 4/2} {\fill[black] (\p,\j) circle (3pt);}
\node[anchor=east]  at (1.75,2) {\small $(2,2)$};
\node[anchor=north] at (3,0.7)  {\small $(3,1)$};
\node[anchor=south] at (3,3.3)  {\small $(3,3)$};
\node[anchor=west]  at (4.25,2) {\small $(4,2)$};
\node[align=center] at (3,-0.9) {\small $W(D,D')=\{(2,1,3,2),\ (2,3,1,2)\}$};
\end{tikzpicture}
\caption{The region $R(D,D')$ for $D=(3,2,3,4)$ and $D'=(5,6,5,4)$. The two linear
extensions give the knitting words $(2,1,3,2)$ and $(2,3,1,2)$, differing by the
commutation of $1$ and $3$.}
\label{fig:region}
\end{figure}

\begin{proof}
A flip at level $j$ moves the path across the next node of $R(D,D')$ on that level. It is legal
exactly when the predecessors of that node have already been crossed. Thus knitting sequences are
precisely the linear extensions of $R(D,D')$, proving (a). Part (b) is immediate, since each flip
at level $j$ raises $d_j$ by $2$.

For (c), any two linear extensions differ by swaps of adjacent incomparable elements. In
$R(D,D')$, nodes on the same or adjacent levels are comparable, so incomparable nodes have levels
$j,j'$ with $|j-j'|>1$. The corresponding generators commute in $\Gam_n$.
\end{proof}

Later, we will use the composition rule
\[
Q(D,D'')=Q(D,D')Q(D',D'')
\qquad(D\le D'\le D''),
\]
which follows by concatenating knitting sequences.

\subsection{Knitting \texorpdfstring{$w_0$}{w0}-words}

\begin{definition}
For a Dynkin path $D$, let $\bar D$ be its image under
$
(i,j)\longmapsto(i+n,n-j),
$
so that
$
\bar d_j=d_{n-j}+n.
$
A \emph{knitting $w_0$-word} for $D$ is a knitting word from $D$ to $\bar D$. We write
\[
\mathcal K(D)=W(D,\bar D),
\qquad
Q_D=Q(D,\bar D)\in\Gam_n.
\]
\end{definition}

The path $\bar D$ has orientation $\alpha(\Omega_D)$ and lies weakly to the right of $D$, since
\[
\bar d_j-d_j=n+d_{n-j}-d_j>0.
\]
Horizontal translation of $D$ translates the whole picture, so $\mathcal K(D)$ and $Q_D$ depend
only on the orientation $\Omega_D$.

\begin{lemma}\label{lem:count-flips}
Every knitting $w_0$-word has length $\binom n2$.
\end{lemma}

\begin{proof}
By Lemma~\ref{lem:knitting-words}(b),
\[
|R(D,\bar D)|
=
\frac12\sum_{j=1}^{n-1}(\bar d_j-d_j)
=
\frac12\sum_{j=1}^{n-1}(n+d_{n-j}-d_j)
=
\binom n2.
\]
\end{proof}

A word $w=(j_1,\dots,j_N)$ is \emph{adapted} to an orientation $\Omega$ if $j_1$ is a source
of $\Omega$ and, successively, $j_r$ is a source after reflections at
$j_1,\dots,j_{r-1}$. This is the usual notion of an adapted reduced word; see, for example,
\cite[\S4.7]{Lusztig}, \cite{Bedard}, and \cite[\S VII]{ASS}.

\begin{lemma}\label{lem:reduced-adapted}
The knitting $w_0$-words for $D$ are exactly the reduced words for $w_0$ adapted to
$\Omega_D$. Moreover, a reduced word for $w_0$ is adapted to at most one orientation.
\end{lemma}

\begin{proof}
Every knitting word is adapted. We first show that it is a reduced word for $w_0$. For the linear path $D_{\mathrm{lin}}=(1,2,\dots,n-1)$, the standard word
\[
(1)(2,1)\cdots(n-1,\dots,1)
\]
is a knitting word. Indeed, after the segments through $(k,\dots,1)$ have been performed, the
heights are
\[
2k+1,\,2k,\,\dots,\,k+2,\,k+1,\,k+2,\dots,n-1,
\]
so $k+1$ is the next source. At the end the $j$-th height is $2n-j$, which is exactly the
$j$-th height of $\bar D_{\mathrm{lin}}$. The word is also the standard reduced word for $w_0$.

Suppose now that $D'$ is obtained from $D$ by flipping a source $j$. Choose
$(j,x)\in\mathcal K(D)$. Then
\[
(x,n-j)\in\mathcal K(D'),
\]
since the flip at $n-j$ takes $\bar D$ to $\bar D'$; see
Lemma~\ref{lem:flip}(1). If $s_js_x=w_0$, then
\[
s_xs_{n-j}=s_jw_0s_{n-j}=w_0.
\]
By Lemma~\ref{lem:sources-exist}, every orientation can be reached from the linear one by source
flips. It follows that some knitting word for every $D$ represents $w_0$. By
Lemma~\ref{lem:knitting-words}(c), they all do. Their length is
$\binom n2=\ell(w_0)$ by Lemma~\ref{lem:count-flips}, so they are reduced.

Conversely, let $w$ and $w'$ be reduced words for the same $u\in\Sym_n$, both adapted to
the same orientation. We claim that they differ only by commutations of nonadjacent letters.

Let $j$ be the first letter of $w$. Since $j$ lies in the support of $u$, it also occurs in
$w'$. Before its first occurrence in $w'$, no neighbour of $j$ can occur: such a neighbour
cannot become a source before the edge joining it to $j$ is reversed, and that requires
reflecting at $j$ first. Hence every preceding letter commutes with $j$, so we may move $j$
to the front of $w'$. Deleting the initial $j$ from both words and reflecting the orientation
at $j$, the claim follows by induction.

Now let $w$ be a reduced word for $w_0$ adapted to $\Omega_D$, and choose
$v\in\mathcal K(D)$. By the claim, $w$ and $v$ differ only by distant commutations. These
commutations preserve legality and the endpoint of the flip sequence, so
$w\in\mathcal K(D)$.

Finally, an adapted reduced word for $w_0$ determines the orientation. Every letter
$1,\dots,n-1$ occurs in such a word, and for each edge $j-(j+1)$, whichever of $j$ and $j+1$
occurs first determines its direction, since no other reflection changes that edge.
\end{proof}

\begin{remark}\label{rem:sinks-sources}
Lemma~\ref{lem:reduced-adapted} is standard in the representation theory of quivers. In this
language, knitting a slice $D$ of the repetition quiver $\ZZ A_{n-1}$ to its translate
$\bar D$ sweeps out the Auslander--Reiten quiver of $\mathrm{rep}(\Omega_D)$. Reading the
levels of its vertices in any linear extension gives a reduced word for $w_0$ adapted to
$\Omega_D$. Conversely, every adapted reduced word arises this way, and the adapted words
for a fixed orientation form a single commutation class. See, for example, Lusztig
\cite[\S4.7]{Lusztig}, B\'edard \cite{Bedard}, and \cite[\S VII]{ASS}. We included the
elementary proof above to keep the discussion accessible.

Much of the literature uses sinks rather than sources. We use sources throughout, as do
Garver, Patrias and Thomas \cite[Thm.~1.10]{GPT}, because a flip of a Dynkin path moves it to the
right. The two conventions are exchanged by reversing all arrows, i.e.\ by replacing $\Omega$
with $\Omega^{\mathrm{op}}$.
\end{remark}

\begin{corollary}\label{cor:Q-linear}
For the linear path $D_{\mathrm{lin}}=(1,2,\dots,n-1)$,
$
Q_{D_{\mathrm{lin}}}=q_{n-1}=q
\;\text{in }\Gam_n.
$
Thus every knitting $w_0$-word for $D_{\mathrm{lin}}$ represents evacuation in $\BK_n$.
\end{corollary}

\begin{proof}
The standard word
\[
(1)(2,1)\cdots(n-1,\dots,1)
\]
is adapted to the linear orientation, so by Lemma~\ref{lem:reduced-adapted} it is a knitting
$w_0$-word for $D_{\mathrm{lin}}$. Since this is the defining word for $q_{n-1}$, we get
$Q_{D_{\mathrm{lin}}}=q_{n-1}=q$. The last statement follows because $q$ acts as evacuation
in $\BK_n$.
\end{proof}

\begin{example}\label{ex:n4-strip}
Let $n=4$ and let $D=(1,2,3)$ be the linear path, so that
$
\Omega_D: 1\to2\to3.
$
Then $\bar D=(7,6,5)$, and $R(D,\bar D)$ consists of the six nodes
\[
(1,1),(3,1),(5,1);\qquad (2,2),(4,2);\qquad (3,3).
\]
They are shown in Figure~\ref{fig:strip}, together with the cover relations. The only
incomparable pair is $\{(3,1),(3,3)\}$, so there are exactly two knitting $w_0$-words:
\[
(1,2,1,3,2,1)\qquad\text{and}\qquad(1,2,3,1,2,1).
\]
These are the two reduced words for $w_0\in\Sym_4$ adapted to $1\to2\to3$, and they differ
by the commutation of $1$ and $3$. The first is the standard word for $q_3$.
\end{example}

\begin{figure}[ht]
\centering
\begin{tikzpicture}[scale=0.72,>=stealth]
\foreach \i in {-1,1,3,5,7,9} {\fill[gray!35] (\i,1) circle (2.2pt);}
\foreach \i in {0,2,4,6,8} {\fill[gray!35] (\i,2) circle (2.2pt);}
\foreach \i in {-1,1,3,5,7,9} {\fill[gray!35] (\i,3) circle (2.2pt);}
\node at (-2.2,1) {$j=1$}; \node at (-2.2,2) {$j=2$}; \node at (-2.2,3) {$j=3$};
\fill[blue!6] (1,1)--(2,2)--(3,3)--(5,3)--(6,2)--(7,1)--cycle;
\draw[very thick,blue] (1,1)--(2,2)--(3,3);
\draw[very thick,red]  (7,1)--(6,2)--(5,3);
\node[blue] at (0.4,0.6) {$D$};
\node[red]  at (7.7,0.6) {$\bar D$};
\draw[thick,->] (1.2,1.2)--(1.8,1.8);   
\draw[thick,->] (2.2,1.8)--(2.8,1.2);   
\draw[thick,->] (2.2,2.2)--(2.8,2.8);   
\draw[thick,->] (3.2,1.2)--(3.8,1.8);   
\draw[thick,->] (3.2,2.8)--(3.8,2.2);   
\draw[thick,->] (4.2,1.8)--(4.8,1.2);   
\foreach \p/\j in {1/1, 3/1, 5/1, 2/2, 4/2, 3/3} {\fill[black] (\p,\j) circle (3pt);}
\node[align=center] at (4,-0.6) {\small $\mathcal K(D)=\{(1,2,1,3,2,1),\ (1,2,3,1,2,1)\}$};
\end{tikzpicture}
\caption{The region $R(D,\bar D)$ for the linear path $D=(1,2,3)$, with cover relations
shown as arrows. The incomparable nodes $(3,1)$ and $(3,3)$ may be flipped in either order,
giving the two knitting $w_0$-words.}
\label{fig:strip}
\end{figure}
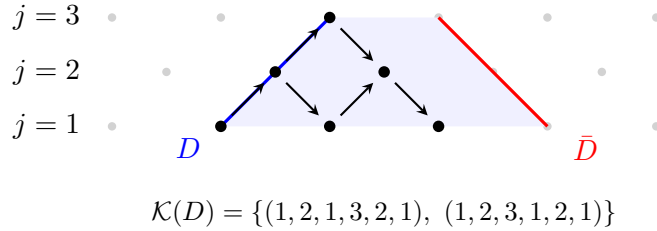

\subsection{Adapted words and sorting networks}\label{sec:X}

A reduced word for $w_0\in\Sym_n$ can be viewed as a \emph{sorting network} on $n$ wires.
Following Warrington \cite{Warrington}, choose a subset of $m$ wires and record the crossings among them.
The resulting network is again a reduced word for $w_0\in\Sym_m$, called an
\emph{$m$-subnetwork}. If $X$ is a set of reduced words in $\Sym_m$, we say that a reduced
word is \emph{$X$-avoiding} if none of its $m$-subnetworks belongs to $X$.

Warrington used this notion to formulate a combinatorial analogue of Sylvester's four-point problem. 
Goodman and Pollack \cite{GoodmanPollack} associate to a configuration of $n$ points in general position an
allowable sequence by rotating a directed line and recording the order of the orthogonal
projections of the points onto it. Each time two points exchange order we record the corresponding
adjacent transposition; over a half-turn every pair exchanges order exactly once, giving a reduced
word for $w_0\in\Sym_n$.  For $n=4$, the sixteen reduced
words split into twelve of \emph{convex type} and four of \emph{reentrant type} \cite[(1) and Fig.~1]{Warrington}; the latter are
\[
X=\{123212,\ 321232,\ 212321,\ 232123\}.
\]

Sylvester's problem asks for the probability that four random points are reentrant. Warrington's
combinatorial analogue asks for the probability that a uniformly random reduced word for
$w_0\in\Sym_n$, together with a uniformly random $4$-subset of its wires, induces a subnetwork in
$X$. He conjectured that this probability is $1/4$ for every $n\ge4$, and Angel and Holroyd
\cite{AH} proved the conjecture.
For $n\ge5$, not every reduced word comes from a point configuration under the
Goodman--Pollack construction. Thus ``convex'' and ``reentrant''  refer here to the
corresponding combinatorial types of reduced words, not necessarily to actual point configurations.

Meng \cite[\S6.3]{Meng} studied $X$-avoiding reduced words for $w_0\in\Sym_n$ and asked for
their number. The following theorem connects this question with adapted words.

\begin{theorem}\label{thm:X-avoiding}
Every knitting $w_0$-word is $X$-avoiding.
\end{theorem}

\begin{proof}
Write $w=(i_1,\dots,i_N)\in\mathcal K(D)$. By
Lemma~\ref{lem:knitting-words}(b), the letter $j$ occurs
\[
m_j=\frac{\bar d_j-d_j}{2}
=\frac{n+d_{n-j}-d_j}{2}
\]
times. Hence
\[
\sum_{r=1}^N(i_r-1)(n-i_r-1)
=
\frac n2\sum_{j=1}^{n-1}(j-1)(n-j-1)
+\frac12\sum_{j=1}^{n-1}(j-1)(n-j-1)(d_{n-j}-d_j).
\]
The second sum is zero by the symmetry $j\leftrightarrow n-j$. One can check that
\[
\sum_{j=1}^{n-1}(j-1)(n-j-1)
=\frac{(n-1)(n-2)(n-3)}6,
\]
so
\[
\sum_{r=1}^N(i_r-1)(n-i_r-1)
=\frac{n(n-1)(n-2)(n-3)}{12}
=2\binom n4.
\]
By \cite[Prop.~5.3]{Meng}, the number of $X$-subnetworks of $w$ is
\[
\sum_{r=1}^N(i_r-1)(n-i_r-1)-2\binom n4,
\]
and is therefore zero.
\end{proof}

\begin{conjecture}\label{conj:X}
A reduced word for $w_0\in\Sym_n$ is $X$-avoiding if and only if it is adapted to an orientation
of $A_{n-1}$. Equivalently, a reduced word for $w_0$ is adapted to an orientation of $A_{n-1}$
if and only if
\[
\sum_r (i_r-1)(n-i_r-1)=2\binom n4.
\]
\end{conjecture}

By Meng's formula, the displayed equality is exactly the condition of being $X$-avoiding, so the
content of the conjecture is that this numerical condition forces adaptedness.
We verified the conjecture by computer for $n=4,5,6$: in these cases the $X$-avoiding reduced
words are exactly the knitting $w_0$-words. For $n=7$ the number of adapted reduced words is
$68641152$, the number of $X$-avoiding reduced words reported by Warrington; by
Theorem~\ref{thm:X-avoiding} the two sets coincide, so the conjecture holds for $n\le7$. We have no conceptual explanation for why adapted reduced words should be governed by a convexity condition on four-point configurations.

\section{Adapted relations}\label{sec:second}

We now show that, modulo the cactus relations, the adapted  relations are equivalent to
the Dynkin-reversal relations of Section~\ref{sec:first}. The key is the following simple observation describing how
the knitting element $Q_D$ changes under a flip.

\begin{lemma}[Flip Lemma]\label{lem:flip}
Let $j$ be a source of a Dynkin path $D$, and let $D'$ be the flip of $D$ at $j$. Then
\begin{enumerate}
\item $n-j$ is a source of $\bar D$, and $\bar{D'}$ is the flip of $\bar D$ at $n-j$;
\item in $\Gam_n$,
\[
Q_{D'}=t_jQ_Dt_{n-j}.
\]
\end{enumerate}
\end{lemma}

\begin{proof}
Part (1) is immediate from
\[
\bar d_k=d_{n-k}+n:
\]
a source at $j$ becomes a source at $n-j$, and flipping there replaces
$\bar d_{n-j}=d_j+n$ by $d_j+n+2$, which is exactly $\bar{D'}$.

For (2), note that $D'\le\bar D$, since the flip raises one height by $2$ and
$\bar d_j-d_j\ge2$. Since $j$ is a source, we can write a knitting word for $D$ as
$(j,x)$ with $x\in W(D',\bar D)$. Thus $Q_D=t_jt_x$, so $t_x=t_jQ_D$. By (1),
$(n-j)\in W(\bar D,\bar{D'})$, and therefore
\[
Q_{D'}=t_xt_{n-j}=t_jQ_Dt_{n-j}.
\]
\end{proof}

\begin{proposition}\label{prop:closed-formula}
Let $u\in W(D_{\mathrm{lin}},D)$. Then
$
Q_D=t_{\rev(u)}\,q\,t_{\alpha(u)}
$
in $\Gam_n$.
\end{proposition}

\begin{proof}
This follows by induction on the length of $u$ from the Flip Lemma, with
Corollary~\ref{cor:Q-linear} as the base case.
\end{proof}

\begin{definition}
The \emph{adapted  relations} are the relations
$Q_D=q$
for every Dynkin path $D$.
Equivalently, they are the relations
$t_w=q$
for every reduced word $w$ for $w_0$ adapted to an orientation of $A_{n-1}$.
\end{definition}

\begin{theorem}\label{thm:equivalence}
In $\Cn_n$, the Dynkin-reversal relations are equivalent to the adapted  relations.
Thus,
\[
\RC_n
=
\Cn_n\Big/\Big\langle\!\Big\langle
Q_D=q\ \text{for all Dynkin paths }D
\Big\rangle\!\Big\rangle .
\]
\end{theorem}

\begin{proof}
Assume first the Dynkin-reversal relations
$q\,t_i\,q=t_{n-i}$, equivalently $t_jq=q\,t_{n-j}$.
If $D'$ is obtained from $D$ by flipping a source $j$ and $Q_D=q$, then the Flip Lemma gives
\[
Q_{D'}=t_jqt_{n-j}=q.
\]
Since $Q_{D_{\mathrm{lin}}}=q$ and every orientation can be reached from the linear one by flips,
we get $Q_D=q$ for every $D$.

In the other direction, assume $Q_D=q$ for every Dynkin path $D$. Choose $D$ with $j$ as a source, and let
$D'$ be its flip at $j$. Then $Q_D=Q_{D'}=q$, so the Flip Lemma gives
$
q=t_jqt_{n-j},
$
or equivalently
$
q\,t_{n-j}\,q=t_j.
$
As $j$ varies, these are exactly the Dynkin-reversal relations.
\end{proof}

\begin{theorem}\label{thm:knit-are-evac}
Every reduced word for $w_0$ adapted to an orientation of $A_{n-1}$ represents evacuation in
$\BKR_n$. Equivalently,
$
Q_D=q
$
in $\BKR_n$ for every Dynkin path $D$.
\end{theorem}

\begin{proof}
The Dynkin-reversal relations hold in $\BKR_n$ by Theorem~\ref{thm:first-kind}, so the adapted
 relations hold by Theorem~\ref{thm:equivalence}.
\end{proof}

\section{Promotion}\label{sec:promotion}

We now relate the adapted  relations to promotion. The key is the following identity,
which already holds in the commutation group.

Recall $p=p_{n-1}=t_{n-1}t_{n-2}\cdots t_1$, so that
$p^{-1}=t_1t_2\cdots t_{n-1}$.

\begin{theorem}\label{thm:promotion-identity}
In $\Gam_n$, and hence in $\Cn_n$, $\BK_n$, and $\BKR_n$,
\[
q_{n-1}\alpha(q_{n-1})
=
(t_1t_2\cdots t_{n-1})^n
=
p^{-n}.
\]
\end{theorem}

\begin{proof}
We compute
\[
Q\bigl(D_{\mathrm{lin}},D_{\mathrm{lin}}+(2n,0)\bigr)
\]
in two ways.

First, $(1,2,\dots,n-1)$ knits $D_{\mathrm{lin}}$ to $D_{\mathrm{lin}}+(2,0)$: after flipping
$1,2,\dots,j$, the next level $j+1$ is a source, and after all $n-1$ flips every height has
increased by $2$. Repeating this $n$ times gives
\[
Q\bigl(D_{\mathrm{lin}},D_{\mathrm{lin}}+(2n,0)\bigr)
=(t_1t_2\cdots t_{n-1})^n.
\]

On the other hand,
\[
D_{\mathrm{lin}}+(2n,0)=\bar{\bar D}_{\mathrm{lin}}.
\]
The map $(i,j)\mapsto(i+n,n-j)$ sends knitting words from $D$ to $D'$ to knitting words from
$\bar D$ to $\bar D'$, replacing every letter $j$ by $n-j$. Therefore 
$
Q(\bar D,\bar D')=\alpha\bigl(Q(D,D')\bigr).
$
Using the composition rule,
\[
Q\bigl(D_{\mathrm{lin}},\bar{\bar D}_{\mathrm{lin}}\bigr)
=
Q(D_{\mathrm{lin}},\bar D_{\mathrm{lin}})
Q(\bar D_{\mathrm{lin}},\bar{\bar D}_{\mathrm{lin}})
=
q\,\alpha(q).
\]
Comparing the two expressions gives the result.
\end{proof}

\begin{corollary}[Haiman's promotion relation]\label{cor:haiman}
The adapted  relations imply
\[
p^n=1.
\]
In particular, promotion has order dividing $n$ on every rectangular shape.
For standard rectangular tableaux this is Haiman's theorem
\cite[Thm.~4.4]{Haiman}; the semistandard statement follows by standardisation,
as observed for example in \cite[Prop.~3.2]{BST}. The order is exactly $n$ in $\BKR_n$.
\end{corollary}

\begin{proof}
For the anti-linear path $\bar D_{\mathrm{lin}}$, we have
$Q_{\bar D_{\mathrm{lin}}}=\alpha(q)$. Thus its adapted  relation is
$\alpha(q)=q$. By Theorem~\ref{thm:promotion-identity},
\[
p^{-n}=q\alpha(q)=q^2=1.
\]
Therefore $p^n=1$.

The order is exactly $n$ in $\BKR_n$, since already on the $1\times1$ rectangle, $p$ acts as an
$n$-cycle.
\end{proof}

\begin{remark}\label{rem:other-proofs}
For standard rectangular tableaux, Haiman's proof uses dual equivalence
\cite[Thm.~4.4]{Haiman}. Several other approaches are known. Bandlow, Schilling and Thi\'ery
\cite[Prop.~3.2]{BST} use crystal operations and \emph{jeu de taquin}; Rhoades
\cite{Rhoades} proves the result representation-theoretically as part of his cyclic-sieving
theorem for rectangular tableaux. Einstein and Propp \cite{EinsteinPropp} interpret
semistandard promotion through piecewise-linear rowmotion on a product of two chains and
deduce its periodicity from the birational rowmotion results of Grinberg and Roby
\cite{GrinbergRoby}. The first author~\cite{FriedenGrassmannian} realizes promotion as the
tropicalization of a cyclic shift on the Grassmannian, and Westbury \cite{Westbury19} uses
the cactus-group action on seminormal bases for Hecke algebras.
\end{remark}

One might ask whether the other way around, Haiman's relation $p^n=1$ implies the Dynkin-reversal relations.
It does not.

\begin{proposition}\label{prop:promotion-weaker}
For every $n\ge4$, the promotion relation $p^n=1$ is strictly weaker than the
Dynkin-reversal relations. Equivalently, the natural map
\[
\Cn_n\big/\langle\!\langle p^n\rangle\!\rangle \twoheadrightarrow \RC_n
\]
is not an isomorphism.
\end{proposition}

\begin{proof}
Any choice of signs $\chi(t_i)\in\{\pm1\}$ defines a homomorphism
$\chi:\Cn_n\to\{\pm1\}$: the signs commute and square to $1$, so all the relations of
Theorem~\ref{thm:CGP} are satisfied.

The Dynkin-reversal relations would force
$\chi(t_i)=\chi(t_{n-i})$. We construct a character satisfying $p^n=1$ but not this relation.

If $n$ is even, set $\chi(t_1)=-1$ and $\chi(t_i)=1$ for $i>1$. Then
$\chi(p^n)=1$, but $\chi(t_1)\ne\chi(t_{n-1})$.

If $n\ge5$ is odd, set $\chi(t_1)=\chi(t_2)=-1$ and $\chi(t_i)=1$ otherwise. Then
$\chi(p)=1$, so again $\chi(p^n)=1$, while
$\chi(t_1)\ne\chi(t_{n-1})$.

The case $n=3$ is an exception: the promotion relation $p^3=1$ is equivalent to the
Dynkin-reversal relation: both are equivalent to
$t_1t_2t_1=t_2t_1t_2$. 
\end{proof}

\section{The big cactus group}\label{sec:big-cactus}

We identify the quotient of Definition~\ref{def:RC} with the big cactus
group. The proof uses the geometry of real moduli spaces and a generation
theorem for a parabolic subgroup of the cactus group.

\subsection{Real Deligne--Mumford spaces and equivariant fundamental groups}
\label{subsec:big-real-moduli}

For $r\geq3$, let $\overline M_{0,r}(\mathbb R)$ be the real
Deligne--Mumford space of stable genus-zero curves with $r$ labeled real
marked points.  A point of $\overline M_{0,r}(\mathbb R)$ is represented by
a connected nodal curve whose irreducible components are copies of
$\mathbb P^1(\mathbb R)$ and such that every component contains at least
three \emph{special points}, meaning marked points or nodes.  Two such
marked curves represent the same point if they are related by projective
automorphisms on their components.

The boundary stratification is described by the dual tree of the stable
curve: vertices correspond to irreducible components, internal edges to
nodes, and leaves to marked points.  The codimension of a stratum is the
number of internal edges of its dual tree.  In particular, the
zero-dimensional strata are represented by trivalent trees; we will use
special zero-dimensional points whose dual trees are caterpillars.

There is one feature of the real locus that will be used repeatedly below.
On every irreducible component, the order of the special points is intrinsic
only up to reversal.  Indeed, after choosing one special point as $\infty$,
the remaining automorphisms have the form
\[
        x\longmapsto ax+b,\qquad a\in\mathbb R^\times,
\]
and $a<0$ reverses the order.  Equivalently, after normalizing the
translation, the residual ambiguity includes the sign change
$x\mapsto -x$.  Thus a real boundary stratum may be encoded by a planar
dual tree, but at each internal vertex the cyclic order of the incident
half-edges is taken modulo reversal, independently at different vertices.

\begin{figure}[ht]
\centering
\begin{tikzpicture}[x=1cm,y=1cm,line cap=round,line join=round]
  \path[use as bounding box] (-0.2,-0.2) rectangle (8.4,3.3);

  \node at (1.8,3.0) {\small a boundary point};

  \def\r{0.62}

  \begin{scope}[shift={(1.25,1.45)}]
    \draw[line width=.85pt] (0,0) circle (\r);
    \fill (210:\r) circle (1.6pt);
    \node[left=2pt] at (210:\r) {$1$};
    \fill (140:\r) circle (1.6pt);
    \node[left=2pt] at (140:\r) {$2$};
  \end{scope}

  \begin{scope}[shift={(2.49,1.45)}]
    \draw[line width=.85pt] (0,0) circle (\r);
    \fill (80:\r) circle (1.6pt);
    \node[above=2pt] at (80:\r) {$3$};
    \fill (350:\r) circle (1.6pt);
    \node[right=0pt] at (350:\r) {$4$};
    \fill (255:\r) circle (1.6pt);
    \node[below=2pt] at (255:\r) {$\infty$};
  \end{scope}

  \fill (1.87,1.45) circle (1.6pt);

  \draw[<->,line width=.7pt] (3.55,1.45)--(4.50,1.45);

  \node at (5.95,3.0) {\small dual tree};

  \coordinate (vL) at (5.40,1.45);
  \coordinate (vR) at (6.45,1.45);

  \draw[line width=.85pt] (vL)--(vR);
  \draw[line width=.85pt] (vL)--(4.85,2.18);
  \draw[line width=.85pt] (vL)--(4.85,0.72);
  \draw[line width=.85pt] (vR)--(7.00,2.18);
  \draw[line width=.85pt] (vR)--(7.28,1.45);
  \draw[line width=.85pt] (vR)--(6.83,0.60);

  \fill (vL) circle (1.6pt);
  \fill (vR) circle (1.6pt);

  \foreach \x/\y/\lab/\pos in {
    4.85/2.18/1/above,
    4.85/0.72/2/below,
    7.00/2.18/3/above,
    7.28/1.45/4/right,
    6.83/0.60/{\infty}/below}
  {
    \fill (\x,\y) circle (1.6pt);
    \node[\pos] at (\x,\y) {$\lab$};
  }
\end{tikzpicture}
\caption{A boundary point of $\overline{M}_{0,n}(\mathbb R)$ and its dual tree.
Each vertex of the tree represents an irreducible component, each internal edge
represents a node, and the labeled leaves represent marked points.}
\label{fig:big-real-boundary}
\end{figure}
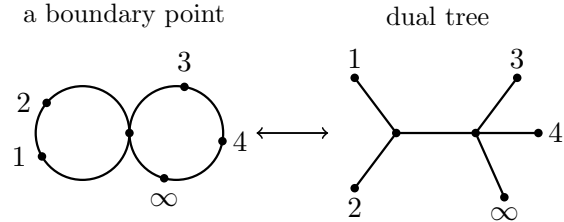

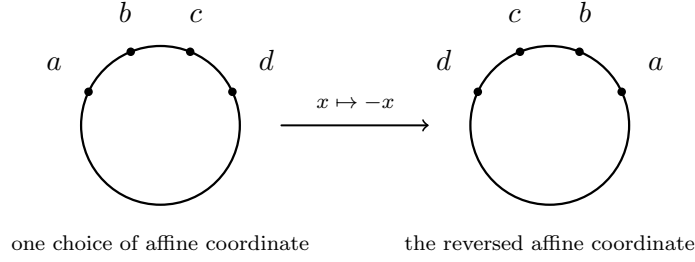
\begin{figure}[ht]
\centering
\begin{tikzpicture}[x=1cm,y=1cm,line cap=round,line join=round]
  \path[use as bounding box] (-0.2,-0.15) rectangle (10.3,3.65);

  \def\R{1.05}

  \begin{scope}[shift={(2.45,2.10)}]
    \draw[line width=.85pt] (0,0) circle (\R);

    \fill (155:\R) circle (1.6pt);
    \node[anchor=south east,xshift=-1pt,yshift=2pt] at (155:1.25) {$a$};

    \fill (112:\R) circle (1.6pt);
    \node[anchor=south,yshift=2pt] at (112:1.25) {$b$};

    \fill (68:\R) circle (1.6pt);
    \node[anchor=south,yshift=2pt] at (68:1.25) {$c$};

    \fill (25:\R) circle (1.6pt);
    \node[anchor=south west,xshift=1pt,yshift=2pt] at (25:1.25) {$d$};
  \end{scope}

  \node[font=\scriptsize] at (2.45,0.55)
    {one choice of affine coordinate};

  \draw[->,line width=.75pt] (4.05,2.10)--(6.00,2.10);
  \node[anchor=south,yshift=2pt] at (5.03,2.10)
    {\scriptsize $x\mapsto -x$};

  \begin{scope}[shift={(7.60,2.10)}]
    \draw[line width=.85pt] (0,0) circle (\R);

    \fill (155:\R) circle (1.6pt);
    \node[anchor=south east,xshift=-1pt,yshift=2pt] at (155:1.25) {$d$};

    \fill (112:\R) circle (1.6pt);
    \node[anchor=south,yshift=2pt] at (112:1.25) {$c$};

    \fill (68:\R) circle (1.6pt);
    \node[anchor=south,yshift=2pt] at (68:1.25) {$b$};

    \fill (25:\R) circle (1.6pt);
    \node[anchor=south west,xshift=1pt,yshift=2pt] at (25:1.25) {$a$};
  \end{scope}

  \node[font=\scriptsize] at (7.60,0.55)
    {the reversed affine coordinate};
\end{tikzpicture}

\caption{On an irreducible real component, the special points are considered
modulo $\operatorname{PGL}_2(\mathbb R)$. Consequently, their cyclic order is
intrinsic only up to reversal.}
\label{fig:big-local-reversal}
\end{figure}

The connected components of the open stratum $M_{0,r}(\mathbb R)$ are
therefore indexed by cyclic orders of the $r$ labels, modulo reversal.  If one
marking is distinguished and denoted by $\infty$, a chamber is equivalently
described by a linear order of the remaining labels, again modulo reversal.
The closure of every chamber is an associahedron, hence is contractible; see
\cite{DJS} and the references therein.

For a finite group $G$ acting on a connected space $X$, we use the
equivariant fundamental group $\pi_1^G(X,x)$.  An element is represented by
a pair $(g,\gamma)$, where $g\in G$ and $\gamma$ is a path from $x$ to
$g x$, modulo homotopy relative to the endpoints.  There is an exact sequence
\begin{equation}\label{eq:big-equivariant-exact}
 1\longrightarrow\pi_1(X,x)\longrightarrow\pi_1^G(X,x)
 \longrightarrow G\longrightarrow1.
\end{equation}
In particular, if $g$ fixes $x$, then $(g,\mathrm{const}_x)$ is an element of
$\pi_1^G(X,x)$; its underlying path is constant although its permutation part
is nontrivial.

We use the topological realization of the cactus group
\cite{DJS,HK}:
\[
 \Cn_n=\pi_1^{\Sym_n}\bigl(\overline M_{0,n+1}(\mathbb R)\bigr),
\]
where the marked points are $1,\ldots,n,\infty$ and $\Sym_n$ fixes
$\infty$.  Under this identification the permutation of $q_{[a,b]}$ is the
reversal $w_{[a,b]}$ of the interval $[a,b]$.  Henriques's \emph{big cactus
group} is
\[
 J'_n=\pi_1^{\Sym_n}\bigl(\overline M_{0,n}(\mathbb R)\bigr),
\]
where all $n$ markings are permuted \cite{HenriquesBig}.

We will also use the \emph{mirabolic cactus group} studied by Gurenkova,
Halacheva and the third author \cite{GHR}:
\[
 \mathrm{MC}_m=
 \pi_1^{\Sym_m}\bigl(\overline M_{0,m+2}(\mathbb R)\bigr)
 \subset \Cn_{m+1}.
\]
Here the markings are $0,1,\ldots,m,\infty$; $\Sym_m$ fixes $0$ and
$\infty$ and permutes $1,\ldots,m$.  Thus $\mathrm{MC}_m$ is the full
preimage of $\Sym_m\subset\Sym_{m+1}$ under the permutation map.  In
particular, the subgroup fixing the first marking in $\Cn_n$ is a relabeled
copy of $\mathrm{MC}_{n-1}$.

\subsection{Caterpillar paths and forgetting the root}
\label{subsec:big-caterpillar-paths}

Fix $n\geq3$ and write
\[
 X_n=\overline M_{0,n+1}(\mathbb R),\qquad
 Y_n=\overline M_{0,n}(\mathbb R).
\]
We use the caterpillar point $x\in X_n$ whose two end components carry
$1,2$ and $n,\infty$, respectively, and whose remaining markings occur in
increasing order along the spine.  Let $D\subset M_{0,n+1}(\mathbb R)$ be
the chamber with cyclic order $1,2,\ldots,n,\infty$.  Since $\overline D$
is contractible, there is a preferred homotopy class of paths
$\gamma\subset\overline D$ from $x$ to a generic point of $D$.  This
identifies the equivariant fundamental group based at the caterpillar orbit
with the usual cactus group $\Cn_n$.

The standard topological representative of the cactus generator
$q_{[a,b]}$ may be seen directly on the open chambers.  Starting from a
configuration in $D$, move the consecutive block $a,a+1,\ldots,b$ onto a
separate component.  At the resulting boundary point the order on that new
real component is defined only up to reversal.  Leaving the boundary in the
other adjacent chamber therefore reverses the block.  The path ends at the
configuration obtained by $w_{[a,b]}$.  Conjugating this path by $\gamma$
gives a representative based at $x$.

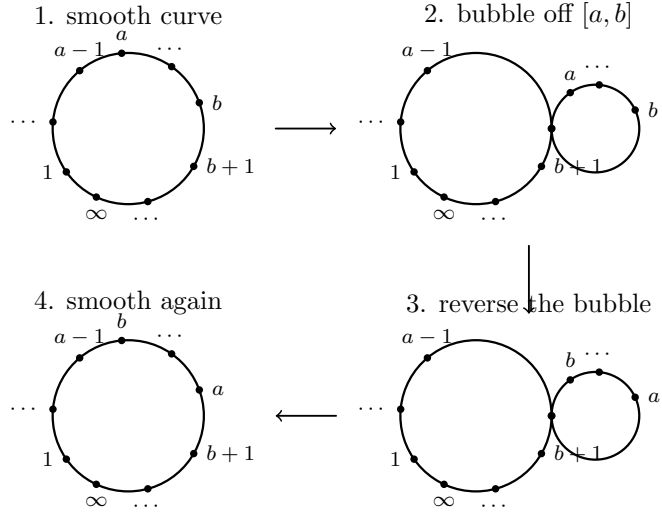
\begin{figure}[ht]
\centering
\begin{tikzpicture}[x=1cm,y=1cm,line cap=round,line join=round]
\tikzset{
  comp/.style={line width=.85pt},
  lab/.style={font=\scriptsize},
  title/.style={font=\small}
}

\def\R{1.00}   
\def\r{0.58}   

\begin{scope}[shift={(0,3.8)}]
  \node[title] at (1.4,2.55) {1. smooth curve};

  \begin{scope}[shift={(1.4,1.05)}]
    \draw[comp] (0,0) circle (\R);

    \fill (215:\R) circle (1.4pt); \node[lab,left=1pt]  at (215:\R) {$1$};
    \fill (175:\R) circle (1.4pt); \node[lab,left=1pt]  at (175:\R) {$\cdots$};
    \fill (130:\R) circle (1.4pt); \node[lab,above=1pt] at (130:\R) {$a-1$};
    \fill (95:\R)  circle (1.4pt); \node[lab,above=1pt] at (95:\R)  {$a$};
    \fill (55:\R)  circle (1.4pt); \node[lab,above=1pt] at (55:\R)  {$\cdots$};
    \fill (20:\R)  circle (1.4pt); \node[lab,right=1pt] at (20:\R)  {$b$};
    \fill (330:\R) circle (1.4pt); \node[lab,right=1pt] at (330:\R) {$b+1$};
    \fill (285:\R) circle (1.4pt); \node[lab,below=1pt] at (285:\R) {$\cdots$};
    \fill (245:\R) circle (1.4pt); \node[lab,below=1pt] at (245:\R) {$\infty$};
  \end{scope}
\end{scope}

\begin{scope}[shift={(4.8,3.8)}]
  \node[title] at (1.9,2.55) {2. bubble off $[a,b]$};

  \begin{scope}[shift={(1.2,1.05)}]
    \draw[comp] (0,0) circle (\R);
    \fill (215:\R) circle (1.4pt); \node[lab,left=1pt]  at (215:\R) {$1$};
    \fill (175:\R) circle (1.4pt); \node[lab,left=1pt]  at (175:\R) {$\cdots$};
    \fill (130:\R) circle (1.4pt); \node[lab,above=1pt] at (130:\R) {$a-1$};
    \fill (330:\R) circle (1.4pt); \node[lab,right=1pt] at (330:\R) {$b+1$};
    \fill (285:\R) circle (1.4pt); \node[lab,below=1pt] at (285:\R) {$\cdots$};
    \fill (245:\R) circle (1.4pt); \node[lab,below=1pt] at (245:\R) {$\infty$};
  \end{scope}

  \begin{scope}[shift={(2.78,1.05)}]
    \draw[comp] (0,0) circle (\r);
    \fill (125:\r) circle (1.4pt); \node[lab,above=1pt] at (125:\r) {$a$};
    \fill (85:\r)  circle (1.4pt); \node[lab,above=1pt] at (85:\r)  {$\cdots$};
    \fill (25:\r)  circle (1.4pt); \node[lab,right=1pt] at (25:\r)  {$b$};
  \end{scope}

  \fill (2.20,1.05) circle (1.5pt);
\end{scope}

\begin{scope}[shift={(4.8,0)}]
  \node[title] at (1.9,2.55) {3. reverse the bubble};

  \begin{scope}[shift={(1.2,1.05)}]
    \draw[comp] (0,0) circle (\R);
    \fill (215:\R) circle (1.4pt); \node[lab,left=1pt]  at (215:\R) {$1$};
    \fill (175:\R) circle (1.4pt); \node[lab,left=1pt]  at (175:\R) {$\cdots$};
    \fill (130:\R) circle (1.4pt); \node[lab,above=1pt] at (130:\R) {$a-1$};
    \fill (330:\R) circle (1.4pt); \node[lab,right=1pt] at (330:\R) {$b+1$};
    \fill (285:\R) circle (1.4pt); \node[lab,below=1pt] at (285:\R) {$\cdots$};
    \fill (245:\R) circle (1.4pt); \node[lab,below=1pt] at (245:\R) {$\infty$};
  \end{scope}

  \begin{scope}[shift={(2.78,1.05)}]
    \draw[comp] (0,0) circle (\r);
    \fill (125:\r) circle (1.4pt); \node[lab,above=1pt] at (125:\r) {$b$};
    \fill (85:\r)  circle (1.4pt); \node[lab,above=1pt] at (85:\r)  {$\cdots$};
    \fill (25:\r)  circle (1.4pt); \node[lab,right=1pt] at (25:\r)  {$a$};
  \end{scope}

  \fill (2.20,1.05) circle (1.5pt);
\end{scope}

\begin{scope}[shift={(0,0)}]
  \node[title] at (1.4,2.55) {4. smooth again};

  \begin{scope}[shift={(1.4,1.05)}]
    \draw[comp] (0,0) circle (\R);

    \fill (215:\R) circle (1.4pt); \node[lab,left=1pt]  at (215:\R) {$1$};
    \fill (175:\R) circle (1.4pt); \node[lab,left=1pt]  at (175:\R) {$\cdots$};
    \fill (130:\R) circle (1.4pt); \node[lab,above=1pt] at (130:\R) {$a-1$};
    \fill (95:\R)  circle (1.4pt); \node[lab,above=1pt] at (95:\R)  {$b$};
    \fill (55:\R)  circle (1.4pt); \node[lab,above=1pt] at (55:\R)  {$\cdots$};
    \fill (20:\R)  circle (1.4pt); \node[lab,right=1pt] at (20:\R)  {$a$};
    \fill (330:\R) circle (1.4pt); \node[lab,right=1pt] at (330:\R) {$b+1$};
    \fill (285:\R) circle (1.4pt); \node[lab,below=1pt] at (285:\R) {$\cdots$};
    \fill (245:\R) circle (1.4pt); \node[lab,below=1pt] at (245:\R) {$\infty$};
  \end{scope}
\end{scope}


\draw[->,line width=.7pt] (3.35,4.85)--(4.15,4.85);

\draw[->,line width=.7pt] (6.70,3.30)--(6.70,2.40);

\draw[->,line width=.7pt] (4.15,1.05)--(3.35,1.05);

\end{tikzpicture}
\caption{A topological representative of the generator $q_{[a,b]}$.
First, the consecutive marked points $a,\ldots,b$ bubble off onto a new
component. Since the real projective line is considered modulo
$\operatorname{PGL}_2(\mathbb R)$, the order on that component is defined only
up to reversal. Reversing the bubbled component and smoothing the node
produces the order obtained from the original one by reversing the interval
$[a,b]$.}
\label{fig:big-qab-path}
\end{figure}

There is a second, boundary-based description which is especially convenient
for us.  Let
\[
 t_1,\ldots,t_{n-1}
\]
be the elementary caterpillar paths of \cite{GHR}.  For $2\leq i<n$, the
path $t_i$ is the shortest edge of the $1$-skeleton joining $x$ to the
adjacent caterpillar obtained by exchanging the positions of $i$ and $i+1$.
Locally, only the four branches
\[
 L=\{1,\ldots,i-1\},\qquad i,\qquad i+1,\qquad
 R=\{i+2,\ldots,n,\infty\}
\]
move.  The corresponding one-dimensional stratum is a copy of
$\overline M_{0,4}(\mathbb R)$: $t_i$ is the arc joining the two displayed
splittings and avoiding the third zero-dimensional point.  For $i=1$, the
first two markings lie on the same end component, so the underlying path is
constant and the equivariant permutation is $(1,2)$.

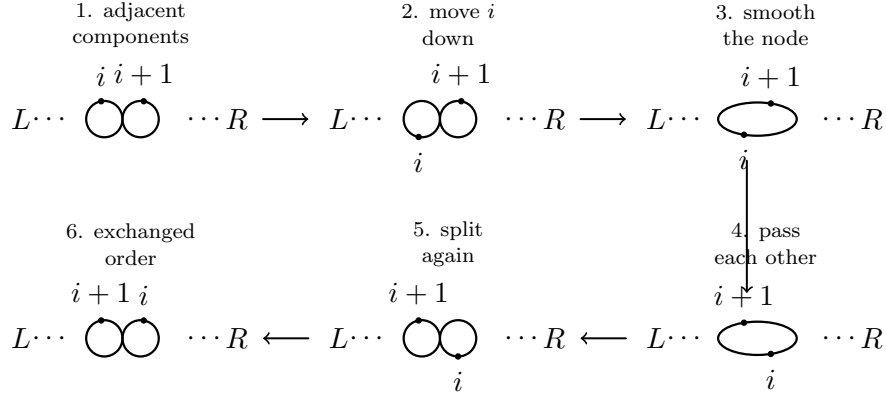
\begin{figure}[ht]
\centering
\begin{tikzpicture}[
    x=1cm,
    y=1cm,
    line cap=round,
    line join=round,
    component/.style={line width=.85pt},
    marked/.style={
        circle,
        fill=black,
        inner sep=0pt,
        minimum size=2.4pt
    },
    steparrow/.style={
        line width=.75pt,
        ->
    }
]

\def\r{0.24}
\def\colsep{4.2}
\def\rowsep{2.9}

\newcommand{\LRparts}{
  \node at (0.25,0) {$L$};
  \node at (0.65,0) {$\cdots$};
  \node at (2.70,0) {$\cdots$};
  \node at (3.10,0) {$R$};
}

\newcommand{\ptitle}[1]{
  \node[font=\scriptsize,align=center] at (1.7,1.25) {#1};
}

\begin{scope}[shift={(0,\rowsep)}]
  \ptitle{1. adjacent\\components}
  \LRparts

  \draw[component] (1.35,0) circle (\r);
  \draw[component] (1.83,0) circle (\r);

  \path (1.35,0) ++(100:\r) coordinate (iA);
  \path (1.83,0) ++(80:\r) coordinate (ipA);

  \node[marked] at (iA) {};
  \node[marked] at (ipA) {};

  \node[anchor=south,yshift=2pt] at (iA) {$i$};
  \node[anchor=south,yshift=2pt] at (ipA) {$i+1$};
\end{scope}

\begin{scope}[shift={(\colsep,\rowsep)}]
  \ptitle{2. move $i$\\down}
  \LRparts

  \draw[component] (1.35,0) circle (\r);
  \draw[component] (1.83,0) circle (\r);

  \path (1.35,0) ++(260:\r) coordinate (iB);
  \path (1.83,0) ++(80:\r) coordinate (ipB);

  \node[marked] at (iB) {};
  \node[marked] at (ipB) {};

  \node[anchor=north,yshift=-2pt] at (iB) {$i$};
  \node[anchor=south,yshift=2pt] at (ipB) {$i+1$};
\end{scope}

\begin{scope}[shift={(2*\colsep,\rowsep)}]
  \ptitle{3. smooth\\the node}
  \LRparts

  \draw[component] (1.59,0) ellipse [x radius=.52,y radius=.22];

  \path (1.59,0) ++({.52*cos(250)},{.22*sin(250)}) coordinate (iC);
  \path (1.59,0) ++({.52*cos(70)},{.22*sin(70)})   coordinate (ipC);

  \node[marked] at (iC) {};
  \node[marked] at (ipC) {};

  \node[anchor=north,yshift=-2pt] at (iC) {$i$};
  \node[anchor=south,yshift=2pt] at (ipC) {$i+1$};
\end{scope}

\begin{scope}[shift={(2*\colsep,0)}]
  \ptitle{4. pass\\each other}
  \LRparts

  \draw[component] (1.59,0) ellipse [x radius=.52,y radius=.22];

  \path (1.59,0) ++({.52*cos(290)},{.22*sin(290)}) coordinate (iD);
  \path (1.59,0) ++({.52*cos(110)},{.22*sin(110)}) coordinate (ipD);

  \node[marked] at (iD) {};
  \node[marked] at (ipD) {};

  \node[anchor=north,yshift=-2pt] at (iD) {$i$};
  \node[anchor=south,yshift=2pt] at (ipD) {$i+1$};
\end{scope}

\begin{scope}[shift={(\colsep,0)}]
  \ptitle{5. split\\again}
  \LRparts

  \draw[component] (1.35,0) circle (\r);
  \draw[component] (1.83,0) circle (\r);

  \path (1.35,0) ++(100:\r) coordinate (ipE);
  \path (1.83,0) ++(270:\r) coordinate (iE);

  \node[marked] at (ipE) {};
  \node[marked] at (iE) {};

  \node[anchor=south,yshift=2pt] at (ipE) {$i+1$};
  \node[anchor=north,yshift=-2pt] at (iE) {$i$};
\end{scope}

\begin{scope}[shift={(0,0)}]
  \ptitle{6. exchanged\\order}
  \LRparts

  \draw[component] (1.35,0) circle (\r);
  \draw[component] (1.83,0) circle (\r);

  \path (1.35,0) ++(100:\r) coordinate (ipF);
  \path (1.83,0) ++(80:\r) coordinate (iF);

  \node[marked] at (ipF) {};
  \node[marked] at (iF) {};

  \node[anchor=south,yshift=2pt] at (ipF) {$i+1$};
  \node[anchor=south,yshift=2pt] at (iF) {$i$};
\end{scope}


\draw[steparrow] (3.45,\rowsep) -- (4.05,\rowsep);
\draw[steparrow] (7.65,\rowsep) -- (8.25,\rowsep);

\draw[steparrow] (9.85,2.35) -- (9.85,0.60);

\draw[steparrow] (8.25,0) -- (7.65,0);
\draw[steparrow] (4.05,0) -- (3.45,0);

\end{tikzpicture}
\caption{A topological representative of the elementary path $t_i$.
Here \(L\) and \(R\) denote the unchanged left and right parts of the
caterpillar curve. The path $t_i$ is realized by moving the marked point
$i$ to the lower side of its component, smoothing the node between the two
adjacent components, passing the two marked points through the smoothed
component, and then splitting the component again. The resulting caterpillar
has the same shape, with the order of $i$ and $i+1$ exchanged.}
\label{fig:big-cactus-generator}
\end{figure}

The comparison theorem of \cite{GHR} identifies
these geometric paths with the elements $t_i$ of
Theorem~\ref{thm:CGP}.  Moreover, if
\[
 p_i=t_i\cdots t_1,\qquad q_i=p_1\cdots p_i,
\]
then the path
\[
 q_{[a,b]}=q_{b-1}q_{b-a}q_{b-1}
\]
is homotopic to the standard interval-reversal path of
Figure~\ref{fig:big-qab-path}.  We will use one further result from
the same paper.  If $\pi:\Cn_n\to\Sym_n$ is the permutation map, then
\begin{equation}\label{eq:big-mirabolic-input}
 \pi^{-1}\bigl(\operatorname{Stab}_{\Sym_n}(1)\bigr)
 =\big\langle q_{[a,b]}\ (2\leq a<b\leq n),\;
                    t_2,\ldots,t_{n-1}\big\rangle.
\end{equation}
This is the mirabolic generation theorem, with the labels shifted by one.

Now let
\[
 f:X_n\longrightarrow Y_n
\]
forget $\infty$ and stabilize, and put $y=f(x)$.  The map is
$\Sym_n$-equivariant, so it induces
\begin{equation}\label{eq:big-forgetful-map}
 f_*:\Cn_n\longrightarrow J'_n.
\end{equation}
For $n\geq4$, forgetting $\infty$ contracts the last component of the rooted
caterpillar.  Thus the two end components of $y$ carry the pairs $1,2$ and
$n-1,n$.  The resulting unrooted caterpillar has a reflection symmetry which,
together with the relabeling $a\mapsto n+1-a$, fixes the point $y$.

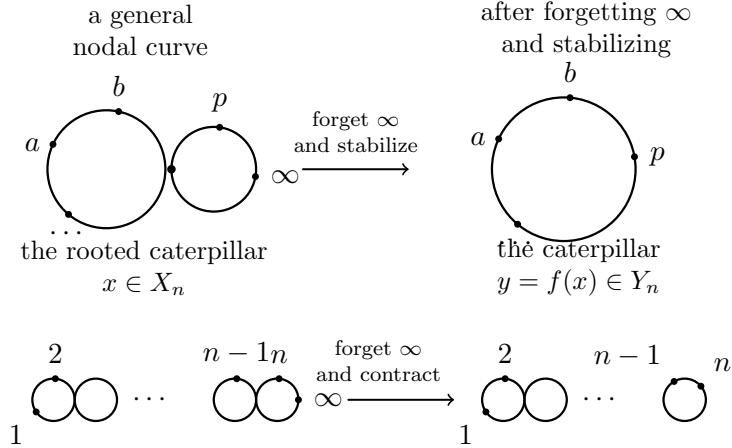
\begin{figure}[ht]
\centering
\begin{tikzpicture}[
    x=1cm,
    y=1cm,
    line cap=round,
    line join=round,
    component/.style={line width=.85pt},
    marked/.style={
        circle,
        fill=black,
        inner sep=0pt,
        minimum size=2.6pt
    },
    maparrow/.style={
        line width=.75pt,
        ->
    }
]

\node[font=\small,align=center] at (2.2,5.20) {a general\\nodal curve};
\node[font=\small,align=center] at (8.1,5.20)
  {after forgetting $\infty$\\and stabilizing};

\begin{scope}[shift={(0.25,3.35)}]

  \draw[component] (1.50,0) circle (0.78);

  \draw[component] (2.92,0) circle (0.56);

  \fill (2.36,0) circle (1.6pt);

  \path (1.50,0) ++(155:0.78) coordinate (aA);
  \path (1.50,0) ++(78:0.78) coordinate (bA);
  \path (1.50,0) ++(230:0.78) coordinate (cA);

  \node[marked] at (aA) {};
  \node[marked] at (bA) {};
  \node[marked] at (cA) {};

  \node[anchor=east,xshift=-1pt] at (aA) {$a$};
  \node[anchor=south,yshift=2pt] at (bA) {$b$};
  \node[anchor=north,yshift=-1pt] at (cA) {$\cdots$};

  \path (2.92,0) ++(82:0.56) coordinate (pA);
  \path (2.92,0) ++(350:0.56) coordinate (infA);

  \node[marked] at (pA) {};
  \node[marked] at (infA) {};

  \node[anchor=south,yshift=2pt] at (pA) {$p$};
  \node[anchor=west,xshift=2pt] at (infA) {$\infty$};

\end{scope}

\draw[maparrow] (4.35,3.35) -- (5.75,3.35);
\node[font=\scriptsize,align=center] at (5.05,3.85)
  {forget $\infty$\\and stabilize};

\begin{scope}[shift={(6.25,3.35)}]

  \draw[component] (1.55,0) circle (0.95);

  \path (1.55,0) ++(155:0.95) coordinate (aB);
  \path (1.55,0) ++(85:0.95) coordinate (bB);
  \path (1.55,0) ++(10:0.95) coordinate (pB);
  \path (1.55,0) ++(230:0.95) coordinate (cB);

  \node[marked] at (aB) {};
  \node[marked] at (bB) {};
  \node[marked] at (pB) {};
  \node[marked] at (cB) {};

  \node[anchor=east,xshift=-1pt] at (aB) {$a$};
  \node[anchor=south,yshift=2pt] at (bB) {$b$};
  \node[anchor=west,xshift=2pt] at (pB) {$p$};
  \node[anchor=north,yshift=-1pt] at (cB) {$\cdots$};

\end{scope}

\node[font=\small,align=center] at (2.25,2.05)
  {the rooted caterpillar\\$x\in X_n$};
\node[font=\small,align=center] at (8.0,2.05)
  {the caterpillar\\$y=f(x)\in Y_n$};

\begin{scope}[shift={(0.10,0.30)}]

  \def\r{0.28}

  \draw[component] (0.95,0) circle (\r);
  \path (0.95,0) ++(215:\r) coordinate (one);
  \path (0.95,0) ++(85:\r) coordinate (two);
  \node[marked] at (one) {};
  \node[marked] at (two) {};
  \node[anchor=north east,xshift=-1pt,yshift=-1pt] at (one) {$1$};
  \node[anchor=south,yshift=2pt] at (two) {$2$};

  \draw[component] (1.51,0) circle (\r);

  \node at (2.25,0) {$\cdots$};

  \draw[component] (3.35,0) circle (\r);
  \path (3.35,0) ++(85:\r) coordinate (nmone);
  \node[marked] at (nmone) {};
  \node[anchor=south,yshift=2pt] at (nmone) {$n-1$};

  \draw[component] (3.91,0) circle (\r);
  \path (3.91,0) ++(85:\r) coordinate (nmark);
  \path (3.91,0) ++(0:\r) coordinate (infmark);
  \node[marked] at (nmark) {};
  \node[marked] at (infmark) {};
  \node[anchor=south,yshift=2pt] at (nmark) {$n$};
  \node[anchor=west,xshift=2pt] at (infmark) {$\infty$};

\end{scope}

\draw[maparrow] (4.95,0.30) -- (6.35,0.30);
\node[font=\scriptsize,align=center] at (5.35,0.82)
  {forget $\infty$\\and contract};

\begin{scope}[shift={(6.05,0.30)}]

  \def\r{0.28}

  \draw[component] (0.95,0) circle (\r);
  \path (0.95,0) ++(215:\r) coordinate (oneB);
  \path (0.95,0) ++(85:\r) coordinate (twoB);
  \node[marked] at (oneB) {};
  \node[marked] at (twoB) {};
  \node[anchor=north east,xshift=-1pt,yshift=-1pt] at (oneB) {$1$};
  \node[anchor=south,yshift=2pt] at (twoB) {$2$};

  \draw[component] (1.51,0) circle (\r);

  \node at (2.25,0) {$\cdots$};

  \draw[component] (3.35,0) circle (\r);
  \path (3.35,0) ++(120:\r) coordinate (nmoneB);
  \path (3.35,0) ++(40:\r) coordinate (nB);
  \node[marked] at (nmoneB) {};
  \node[marked] at (nB) {};
  \node[anchor=south east,xshift=-1pt,yshift=2pt] at (nmoneB) {$n-1$};
  \node[anchor=south west,xshift=1pt,yshift=2pt] at (nB) {$n$};

\end{scope}

\end{tikzpicture}
\caption{Forgetting the root marked point $\infty$.
Top: if a terminal component carries $\infty$ and one additional marked point
\(p\), then after forgetting $\infty$ that component becomes unstable and is
contracted. Bottom: for the rooted caterpillar \(x\in X_n\), the last component
carries \(n\) and \(\infty\); forgetting \(\infty\) contracts this component and
produces the ordinary caterpillar \(y=f(x)\in Y_n\).}
\label{fig:big-forgetting-root}
\end{figure}

Recall $q=q_{[1,n]}=q_{n-1}$, and set
\begin{equation}\label{eq:big-right-anchored}
 t_i^{\mathrm R}
 =q_{[i+1,n]}q_{[i,n]}q_{[i+1,n]}q_{[i+2,n]}
 \qquad(1\leq i<n),
\end{equation}
with a singleton or empty interval interpreted as the identity.  The cactus
relations and Theorem~\ref{thm:CGP} give
\begin{equation}\label{eq:big-anchor-conjugation}
 qt_iq=t_{n-i}^{\mathrm R}.
\end{equation}

\begin{lemma}\label{lem:big-projected-paths}
For $1\leq i<n$,
\[
 f_*(qt_iq)=f_*(t_{n-i});
 \qquad\text{equivalently,}\qquad
 f_*(t_i)=f_*(t_i^{\mathrm R}).
\]
\end{lemma}

\begin{proof}
The crucial observation is that the unrooted caterpillar
$y=f(x)$ is fixed by the longest permutation $w_0(a)=n+1-a$.
Indeed, after forgetting $\infty$, the last component of
$x$ becomes unstable and is contracted. The resulting
caterpillar is symmetric under reversal of the chain:
reflecting the chain and reversing the order of special
points on each component gives an isomorphism of the
corresponding marked curves, since the order on each real
component is defined only up to reversal. Thus
\[
w_0y=y.
\]
Let $r=(w_0,\operatorname{const}_y)\in J'_n$ be the constant-path lift of $w_0$.

Since $w_0$ reverses the cyclic order $1,\ldots,n,\infty$, it maps the chamber $D$ to
itself, and the full cactus reversal $q=q_{[1,n]}$ is represented by a path from $x$ to
$w_0x$ contained in $\overline D$.
After forgetting $\infty$, this path becomes a loop
at $y$ contained in the closure of the corresponding
chamber of $\overline M_{0,n}(\mathbb R)$.
Since this chamber closure is contractible, the loop
is null-homotopic. Hence
\[
        f_*(q)=r.
\]

For $2\leq i\leq n-2$, the projected path $f(t_i)$ varies only the
four-pointed component with branches
\[
 L=\{1,\ldots,i-1\},\qquad i,\qquad i+1,\qquad
 R=\{i+2,\ldots,n\}.
\]
It joins the splittings
$(L,i)\mid(i+1,R)$ and $(L,i+1)\mid(i,R)$ along the arc of
$\overline M_{0,4}(\mathbb R)$ avoiding
$(L,R)\mid(i,i+1)$.  Reflection exchanges the two outer branches and, after
relabeling by $w_0$, carries this arc to the corresponding arc for $n-i$.
Therefore
\[
        r f_*(t_i)r=f_*(t_{n-i}).
\]
At the two ends, $f_*(t_1)$ is the constant-path transposition of the first
cherry, while $f(t_{n-1})$ becomes constant after the component carrying
$n,\infty$ is contracted; its equivariant permutation is the transposition of
the last cherry.  Reflection exchanges these two cases.  For $n=3$ the same
statement is simply the corresponding identity in $\Sym_3$.  Combining this
with $f_*(q)=r$ proves
\[
        f_*(qt_iq)=f_*(t_{n-i})
\]
for every $1\leq i<n$.  The equivalent form follows from
\eqref{eq:big-anchor-conjugation}.
\end{proof}

\subsection{The defining relations}

\begin{theorem}\label{thm:big-cactus}
For $n\geq3$, forgetting $\infty$ induces an isomorphism
\[
 \RC_n\xrightarrow{\sim}J'_n.
\]
Equivalently, the map~\eqref{eq:big-forgetful-map} is surjective and
\begin{equation}\label{eq:big-kernel}
 \ker f_*=
 \big\langle\!\big\langle
   qt_iq\,t_{n-i}:1\leq i<n
 \big\rangle\!\big\rangle.
\end{equation}
\end{theorem}

\begin{proof}
There is a section $\sigma_n:Y_n\to X_n$ of $f$: attach at the marking
$n$ a three-pointed component carrying $n$, $\infty$, and the node.
It is equivariant for $\Sym_{n-1}=\operatorname{Stab}_{\Sym_n}(n)$,
and $\sigma_n(y)=x$. Set
\[
 H=\langle q_{[a,b]}:1\leq a<b\leq n-1\rangle\subset\Cn_n,
 \qquad
 K=\pi_1^{\Sym_{n-1}}(Y_n,y)\subset J'_n.
\]
Regarding $n$ as the root identifies $K$ with $\Cn_{n-1}$.
The section identifies its image with the standard subgroup $H$.
Since $f\sigma_n=\id$, we have an isomorphism
\begin{equation}\label{eq:big-section-isomorphism}
 f_*|_H:H\xrightarrow{\sim}K.
\end{equation}
In particular, the image of $f_*$ contains the pure subgroup of $J'_n$
and projects onto $\Sym_n$, so $f_*$ is surjective.

By Lemma~\ref{lem:big-projected-paths}, write $f_*=\bar f\eta$, where
$\eta:\Cn_n\to\RC_n$ is the quotient map. Conjugating
\eqref{eq:big-mirabolic-input} by $q$ gives
\[
 P:=\pi^{-1}(\Sym_{n-1})
   =\langle H,t_1^{\mathrm R},\ldots,t_{n-2}^{\mathrm R}\rangle.
\]
In $\RC_n$ we have $\eta(t_i^{\mathrm R})=\eta(t_i)$.
Moreover, $t_i\in H$ for $i\leq n-2$ by Theorem~\ref{thm:CGP}.
Hence $\eta(P)=\eta(H)$.

The permutation map descends to $\RC_n$, since all the additional
relators have trivial permutation. If $g\in\ker\bar f$, any lift of
$g$ to $\Cn_n$ has trivial permutation and therefore belongs to $P$.
Thus $g=\eta(h)$ for some $h\in H$. Now $f_*(h)=\bar f(g)=1$, and
\eqref{eq:big-section-isomorphism} implies $h=1$. Therefore
$\bar f$ is injective.
\end{proof}

\begin{corollary}\label{cor:big-two-presentations}
The big cactus group has a presentation on $t_1,\ldots,t_{n-1}$ with
relations
\[
 \begin{gathered}
 t_i^2=1,\qquad t_it_j=t_jt_i\quad(|i-j|>1),\\
 (t_iq_{[j,k]})^2=1\quad(i+1<j<k),\\
 qt_iq=t_{n-i}\quad(1\leq i<n),
 \end{gathered}
\]
where $q_i$ and $q_{[j,k]}$ are the words of Theorem~\ref{thm:CGP}.
A second presentation is obtained by replacing the last family by
$t_w=q$ for every reduced word $w$ for $w_0\in\Sym_n$ adapted to an
orientation of $A_{n-1}$.
\end{corollary}

\begin{proof}
Combine Theorem~\ref{thm:big-cactus} with
Theorems~\ref{thm:CGP} and~\ref{thm:equivalence}.
\end{proof}

\begin{remark}
Davis--Januszkiewicz--Scott give a geometric presentation of the big
cactus group in \cite[Theorem~8.6.2 and Remark~8.6.5]{DJS}, in terms
of polygonal subdivisions and dihedral symmetries (in their notation,
the group is $\operatorname{Aut}(X)$). The presentations above instead
use the elements $t_i$ and describe precisely the relations introduced
by forgetting a marking.
\end{remark}

\section{Invariant tensors}\label{sec:invariants}

We now prove Theorem~\ref{thm:inv-adapted}. The main step is the Dynkin-reversal relation on
invariants.

\begin{theorem}\label{thm:inv-dynkin}
On $\operatorname{Inv}(B_1\otimes\cdots\otimes B_n)$,
\[
q_{[1,n]}t_iq_{[1,n]}=t_{n-i},
\qquad 1\le i\le n-1.
\]
\end{theorem}

Theorem~\ref{thm:inv-adapted} follows from Theorem~\ref{thm:equivalence} and
Proposition~\ref{prop:inv-reversal}: the Dynkin-reversal relations give
$Q_D=q_{[1,n]}$ for every Dynkin path $D$, while $q_{[1,n]}$ acts on invariants by
\[
b_1\otimes\cdots\otimes b_n
\longmapsto
\xi(b_n)\otimes\cdots\otimes\xi(b_1).
\]

For the combinatorial proof of Theorem~\ref{thm:inv-dynkin} in type $A$, we refine the factors into single boxes and factor the two local rules into the same sequence of pairwise box interchanges. Proposition~\ref{prop:inv-atomic} then gives the result one crossing at a time. Section~\ref{subsec:inv-topological} gives a second proof, valid in every semisimple type.

\subsection{Crystals, the Sch\"utzenberger involution, and the commutor}\label{sec:inv-crystals}

Fix $\mathfrak{sl}_m$. A (finite, normal) \emph{crystal} is a finite set $B$ equipped with a weight
map $\wt\colon B\to\ZZ^m$ and partial operators
\[
e_i,f_i\colon B\to B,\qquad 1\le i<m,
\]
such that $f_i(b)=b'$ if and only if $e_i(b')=b$, and
\[
\wt(f_i b)=\wt(b)-\alpha_i.
\]
We will only consider normal crystals, so every connected component is the crystal $B(\lambda)$
of an irreducible representation. The standard model is
$
B(\lambda)=\SSYT(\lambda,m),
$
and in particular $B(\varpi_1)=\{1,\dots,m\}$. We write
\[
\varepsilon_i(b)=\max\{k:e_i^k b\text{ is defined}\},
\qquad
\varphi_i(b)=\max\{k:f_i^k b\text{ is defined}\}.
\]
By a \emph{morphism} of crystals we mean a \emph{strict morphism}:
a map $g\colon B\to B'$ satisfying
\[
\wt(g(b))=\wt(b),\qquad
g(e_i b)=e_i g(b),\qquad
g(f_i b)=f_i g(b)
\]
for every $b\in B$ and every simple root index $i$.
Here we adjoin a symbol $0$ to each crystal, interpret every undefined
crystal operator as having value $0$, and extend $g$ by $g(0)=0$.
Thus the identities above require preservation of undefinedness as well.
The image of a morphism is a union of connected components of $B'$;
in particular, an injective morphism identifies $B$ with such a union.
Thus, if $g\colon B\to B'$ is an injective morphism, its image is a union of connected components of $B'$.

An element is \emph{highest weight} if every $e_i$ is undefined, and \emph{lowest weight} if every
$f_i$ is undefined. We call it \emph{invariant} if it is both, or equivalently if it forms a
one-element connected component.

\begin{definition}[tensor product]\label{conv:tensor}
The tensor product $B\otimes C$ is the set $B\times C$, with
\[
\wt(b\otimes c)=\wt(b)+\wt(c),
\]
and
\[
f_i(b\otimes c)=
\begin{cases}
f_i(b)\otimes c,&\varepsilon_i(b)\ge\varphi_i(c),\\
b\otimes f_i(c),&\text{otherwise,}
\end{cases}
\]
while
\[
e_i(b\otimes c)=
\begin{cases}
e_i(b)\otimes c,&\varepsilon_i(b)>\varphi_i(c),\\
b\otimes e_i(c),&\text{otherwise.}
\end{cases}
\]
This convention is associative, and in a highest weight element of $B\otimes C$ the second factor
is highest weight.
\end{definition}

The ordinary flip
$
b\otimes c\longmapsto c\otimes b
$
is not a crystal morphism $B\otimes C\to C\otimes B$, although the two crystals are isomorphic.
For an irreducible crystal $B(\lambda)$, the \emph{Sch\"utzenberger involution} $\xi$ is the unique
bijection satisfying
\[
\xi(b_\lambda)=b_\lambda^{\mathrm{low}},
\qquad
\xi(e_i b)=f_{m-i}\xi(b),
\]
where $b_\lambda$ and $b_\lambda^{\mathrm{low}}$ are the highest and lowest weight elements.
It is an involution, and on $\SSYT(\lambda,m)$ it is the usual Sch\"utzenberger evacuation
\cite{HK,LascouxSchutzenberger}; by Lemma~\ref{lem:q-is-evac}, this is the Bender--Knuth word $q_{m-1}$. For a general finite crystal, we define $\xi$ componentwise.

The Sch\"utzenberger involution is natural with respect to crystal isomorphisms:
\[
g\circ\xi_B=\xi_{B'}\circ g
\]
for every isomorphism $g\colon B\to B'$. Indeed, $g\xi_Bg^{-1}$ sends the highest weight element
of each component to its lowest weight element and intertwines $e_i$ with $f_{m-i}$, so it must
equal $\xi_{B'}$. In particular, $\xi$ fixes every invariant, since on a one-element component
the highest and lowest weight elements are the same.

\begin{definition}[Henriques--Kamnitzer \cite{HK}]\label{def:commutor}
The \emph{crystal commutor}
\[
\sigma_{B,C}\colon B\otimes C\longrightarrow C\otimes B
\]
is defined by
\[
\sigma_{B,C}(b\otimes c)
=
\xi_{C\otimes B}\bigl(\xi_C(c)\otimes\xi_B(b)\bigr).
\]
\end{definition}

\begin{example}
For $\mathfrak{sl}_2$ write $B(1)=\{1,2\}$ and $B(2)=\{11,12,22\}$, with $\xi(1)=2$,
$\xi(11)=22$ and $\xi(12)=12$. The crystal $B(1)\otimes B(2)$ has components
\[
  \{\,1\otimes11,\ 1\otimes12,\ 1\otimes22,\ 2\otimes22\,\}\cong B(3),\qquad
  \{\,2\otimes11,\ 2\otimes12\,\}\cong B(1),
\]
and $B(2)\otimes B(1)$ has components
$\{11\otimes1,\,11\otimes2,\,12\otimes2,\,22\otimes2\}\cong B(3)$ and
$\{12\otimes1,\,22\otimes1\}\cong B(1)$. Unwinding Definition~\ref{def:commutor}, for instance
$\sigma(1\otimes22)=\xi_{B(2)\otimes B(1)}(\xi(22)\otimes\xi(1))=\xi(11\otimes2)=12\otimes2$, one finds
\[
\begin{array}{c|cccccc}
 b\otimes c & 1\otimes11 & 1\otimes12 & 1\otimes22 & 2\otimes22 & 2\otimes11 & 2\otimes12\\ \hline
 \sigma(b\otimes c) & 11\otimes1 & 11\otimes2 & 12\otimes2 & 22\otimes2 & 12\otimes1 & 22\otimes1
\end{array}
\]
Compare the flip: it sends $1\otimes22$ to $22\otimes1$, which lies in the wrong component. Here
$B(1)\otimes B(2)$ is multiplicity-free, so $\sigma$ is the \emph{only} isomorphism
$B(1)\otimes B(2)\to B(2)\otimes B(1)$; in general there are several, and the content of
Definition~\ref{def:commutor} is the choice of one.
\end{example}

Thus the ordinary flip is corrected by applying the Sch\"utzenberger involution to the two
factors, and then once more to the resulting tensor product. The two applications of $\xi$
reverse the crystal operators twice, so the resulting map is a crystal isomorphism. We will use
the following properties, proved in \cite{HK}.

\begin{proposition}\label{prop:coboundary}
\begin{enumerate}
\item[(a)] The commutor is involutive:
\[
\sigma_{C,B}\circ\sigma_{B,C}=\id.
\]

\item[(b)] It is natural with respect to isomorphisms and inclusions of unions of connected
components. Thus, if $g\colon B\to B'$ and $h\colon C\to C'$ are maps of either kind, then
\[
\sigma_{B',C'}\circ(g\otimes h)
=
(h\otimes g)\circ\sigma_{B,C}.
\]

\item[(c)] It satisfies the cactus axiom:
\[
\sigma_{B,\,D\otimes C}\circ(1_B\otimes\sigma_{C,D})
=
\sigma_{C\otimes B,\,D}\circ(\sigma_{B,C}\otimes1_D).
\]

\item[(d)] In general it does not satisfy the hexagon axiom
\[
\sigma_{B\otimes C,D}
=
(\sigma_{B,D}\otimes1_C)\circ(1_B\otimes\sigma_{C,D}).
\]
\end{enumerate}
\end{proposition}

Henriques and Kamnitzer state (b) for isomorphisms. The extension to inclusions of unions of
connected components follows immediately from the definition, since the Sch\"utzenberger
involution is defined componentwise.

A monoidal category with a natural involutive commutor satisfying the cactus axiom is called a
\emph{coboundary category}. Thus crystals form a coboundary category, but not a symmetric one.

The cactus axiom says, in particular, that there is a canonical way to reverse any number of
tensor factors. For three factors, reversing $C,D$ and then moving $B$ to the end gives the same
map as reversing $B,C$ and then moving $D$ to the front. Iterating this construction gives a
canonical full reversal
\[
\rho\colon X_1\otimes\cdots\otimes X_k
\longrightarrow
X_k\otimes\cdots\otimes X_1.
\]
If
\[
X=X_1\otimes\cdots\otimes X_j,
\qquad
Y=X_{j+1}\otimes\cdots\otimes X_k,
\]
then
\begin{equation}\label{eq:inv-rho}
\rho_{X\otimes Y}
=
(\rho_Y\otimes\rho_X)\circ\sigma_{X,Y}.
\end{equation}
By the cactus axiom this is independent of the choice of splitting. Moreover, $\rho$ is an
involution, and for two factors $\rho_{B\otimes C}=\sigma_{B,C}$.

\subsection{The cactus group action}\label{sec:inv-action}

The cactus group $\Cn_n$ acts on tensor products of $n$ crystals. The generator $q_{[i,j]}$
acts by the full reversal $\rho$ on the factors in positions $i,\dots,j$, and as the identity on
the remaining factors. Thus it maps
$
B_1\otimes\cdots\otimes B_n
$
to the tensor product in which the factors $B_i,\dots,B_j$ appear in reverse order. More
precisely, this is an action on the groupoid of all orderings of the factors; when all $B_i$ are
equal, it is an ordinary action on a single crystal.

The cactus relations (C1)--(C3) follow from the properties of the commutor \cite{HK}. Relation
(C1) follows because $\rho$ is an involution, and (C2) because reversals of disjoint intervals act
on disjoint factors. For (C3), reversing $[i,j]$ carries the interval $[k,l]$ to
$
[i+j-l,\,i+j-k],
$
and the independence of the splitting in \eqref{eq:inv-rho} shows that conjugating the reversal of
$[k,l]$ gives the reversal of this new interval. In particular, the elements $t_i\in\Cn_n$ from
Theorem~\ref{thm:CGP} act on tensor products, with
\[
t_i^2=1,
\qquad
t_it_j=t_jt_i\quad\text{for }|i-j|>1.
\]

We will also need to express the commutor itself in terms of interval reversals. Suppose that
$X$ occupies positions $[p,r']$ and $Y$ occupies positions $[r'+1,r]$. Then
\begin{equation}\label{eq:inv-sigma-positions}
\sigma_{X,Y}
=
q_{[p,\,p+r-r'-1]}\,
q_{[p+r-r',\,r]}\,
q_{[p,r]}
=
q_{[p,r]}\,
q_{[p,r']}\,
q_{[r'+1,r]}.
\end{equation}
The first expression reverses the whole block and then reverses $X$ and $Y$ separately in their
new positions; the second reverses the two blocks first and then reverses the whole interval.
A one-element interval is understood to act trivially.

\begin{lemma}[naturality of the action]\label{lem:inv-natural}
Suppose that each $g_i\colon B_i\to B_i'$ is an isomorphism or an inclusion of a union of
connected components, and let $w\in\Cn_n$ permute the positions by $\pi$. Then
\[
w\circ(g_1\otimes\cdots\otimes g_n)
=
(g_{\pi^{-1}(1)}\otimes\cdots\otimes g_{\pi^{-1}(n)})\circ w.
\]
\end{lemma}

\begin{proof}
Each interval reversal is obtained from commutors and identity maps using
\eqref{eq:inv-rho}, and the commutor is natural.
\end{proof}

\begin{example}
For $n=4$ and $i=2$, the four factors of
$
t_2=q_{[1,2]}q_{[1,3]}q_{[1,2]}q_{[1,1]}
$
act, from right to left, as
\[
  B_1B_2B_3B_4
  \xrightarrow{\,q_{[1,1]}\,}
  B_1B_2B_3B_4
  \xrightarrow{\,q_{[1,2]}\,}
  B_2B_1B_3B_4
  \xrightarrow{\,q_{[1,3]}\,}
  B_3B_1B_2B_4
  \xrightarrow{\,q_{[1,2]}\,}
  B_1B_3B_2B_4.
\]
Thus $B_2$ first moves past $B_1$ to the front, and then past $B_1\otimes B_3$ to the back.

Similarly, the right-anchored element
$
t'_2=q_{[3,4]}q_{[2,4]}q_{[3,4]}q_{[4,4]}
$
acts by
\[
B_1B_2B_3B_4
\xrightarrow{\,q_{[4,4]}\,}
B_1B_2B_3B_4
\xrightarrow{\,q_{[3,4]}\,}
B_1B_2B_4B_3
\xrightarrow{\,q_{[2,4]}\,}
B_1B_3B_4B_2
\xrightarrow{\,q_{[3,4]}\,}
B_1B_3B_2B_4.
\]
Here $B_3$ first moves past $B_4$ to the back, and then past $B_2\otimes B_4$ to the front.
Both maps therefore interchange $B_2$ and $B_3$, ending at $B_1B_3B_2B_4$.
Theorem~\ref{thm:inv-dynkin} says that on invariants these two maps are equal.
\end{example}

\begin{proposition}[reversal on invariants]\label{prop:inv-reversal}
On $\operatorname{Inv}(B_1\otimes\cdots\otimes B_n)$, the full reversal $q_{[1,n]}$ acts by
\[
b_1\otimes\cdots\otimes b_n
\longmapsto
\xi(b_n)\otimes\cdots\otimes\xi(b_1).
\]
\end{proposition}

\begin{proof}
Write
\[
\Theta(b_1\otimes\cdots\otimes b_n)
=
\xi(b_n)\otimes\cdots\otimes\xi(b_1).
\]
Iterating the recursive definition of the full reversal, using naturality of $\xi$ and the fact
that $\Theta$ intertwines $e_i$ with $f_{m-i}$, and hence conjugates the Sch\"utzenberger
involution to the Sch\"utzenberger involution on the reversed tensor product, gives
\[
q_{[1,n]}=\xi_{B_n\otimes\cdots\otimes B_1}\circ\Theta.
\]
Now let $x\in\operatorname{Inv}(B_1\otimes\cdots\otimes B_n)$. Since $q_{[1,n]}$ is a crystal
isomorphism, $q_{[1,n]}x$ is again invariant, now in
$B_n\otimes\cdots\otimes B_1$. The Sch\"utzenberger involution fixes every invariant. Therefore
$
q_{[1,n]}x
=
\xi\bigl(\Theta(x)\bigr)
$
implies
\[
\Theta(x)
=
\xi\bigl(q_{[1,n]}x\bigr)
=
q_{[1,n]}x,
\]
where in the first equality we use that $\xi$ is an involution, and in the second that
$q_{[1,n]}x$ is invariant. Therefore on invariants the full reversal is exactly
\[
b_1\otimes\cdots\otimes b_n
\longmapsto
\xi(b_n)\otimes\cdots\otimes\xi(b_1).
\]
\end{proof}

This is the general form of Lemma~\ref{lem:theta-is-evac}: on invariants, the full cactus reversal
is simply reversal of the factors together with the Sch\"utzenberger involution on each factor.

\subsection{The Bender--Knuth elements as local rules}\label{sec:inv-local}

Fix $1\le i<n$, and write
\[
A=B_1\otimes\cdots\otimes B_{i-1},\qquad
B=B_i,\qquad
C=B_{i+1},\qquad
D=B_{i+2}\otimes\cdots\otimes B_n,
\]
where an empty tensor product is the unit. We use the convention that the reversal of a
one-element or empty interval is the identity:
\[
q_{[p,p]}=q_{[p,p-1]}=1.
\]
Then, uniformly for all $1\le i<n$,
\[
t_i
=
q_{[1,i]}\,q_{[1,i+1]}\,q_{[1,i]}\,q_{[1,i-1]}.
\]
Define the right-anchored analogue
\[
t'_i
:=
q_{[i+1,n]}\,q_{[i,n]}\,q_{[i+1,n]}\,q_{[i+2,n]}.
\]

The point is that $t_i$ and $t'_i$ give two ways to interchange the adjacent factors $B$ and $C$:
$t_i$ is anchored on the prefix $A$, while $t'_i$ is anchored on the suffix $D$.

\begin{lemma}[local-rule form]\label{lem:inv-local}
As maps
\[
A\otimes B\otimes C\otimes D
\longrightarrow
A\otimes C\otimes B\otimes D,
\]
we have
\begin{align}
t_i
&=
\bigl[
\sigma_{B,\,A\otimes C}\circ(\sigma_{A,B}\otimes1_C)
\bigr]\otimes1_D,
\label{eq:inv-tau}\\
t'_i
&=
1_A\otimes
\bigl[
\sigma_{B\otimes D,\,C}\circ(1_B\otimes\sigma_{C,D})
\bigr].
\label{eq:inv-taup}
\end{align}
Moreover,
$
q_{[1,n]}t_iq_{[1,n]}=t'_{n-i}.
$
Thus Theorem~\ref{thm:inv-dynkin} is equivalent to
$
t_i=t'_i,\; 1\le i<n.
$
\end{lemma}

\begin{proof}
We use the uniform expression
$
t_i=q_{[1,i]}q_{[1,i+1]}q_{[1,i]}q_{[1,i-1]},
$
with the rightmost factor acting first. By \eqref{eq:inv-sigma-positions},
$\sigma_{A,B}\otimes1_C$ interchanges the blocks $[1,i-1]$ and $[i,i]$, and thus
\[
\sigma_{A,B}\otimes1_C
=
q_{[2,i]}q_{[1,i]}
=
q_{[1,i]}q_{[1,i-1]}.
\]
The second equality is (C3) when the intervals are nondegenerate, and follows directly from our
convention at the boundary. These are precisely the two rightmost factors in $t_i$.

After this interchange the factors are ordered as $B,A,C$. The map
$\sigma_{B,A\otimes C}$ moves $B$ past the block $A\otimes C$, so again by
\eqref{eq:inv-sigma-positions},
\[
\sigma_{B,A\otimes C}
=
q_{[1,i]}q_{[1,i+1]}.
\]
These are the two remaining factors in $t_i$, proving \eqref{eq:inv-tau}.
The proof of \eqref{eq:inv-taup} is the same argument from the right.

Finally, conjugation by $q_{[1,n]}$ reverses the interval labels:
$
q_{[1,n]}q_{[1,j]}q_{[1,n]}
=
q_{[n+1-j,n]}.
$
Thus conjugating the four factors in $t_i$ gives
$
q_{[1,n]}t_iq_{[1,n]}=t'_{n-i}.
$
\end{proof}

Westbury \cite[Def.~3.1]{Westbury18} interprets the Bender--Knuth generators of the cactus group as
local rules in a coboundary category. In his notation,
\[
\tau^A_{B,C}
:=
\sigma_{B,A\otimes C}\circ(\sigma_{A,B}\otimes1_C).
\]
We will also use the right-anchored version
\[
\bar\tau^{\,D}_{B,C}
:=
\sigma_{B\otimes D,C}\circ(1_B\otimes\sigma_{C,D}).
\]
Thus $t_i$ acts by the local rule $\tau^A_{B,C}$, while $t'_i$ acts by the corresponding
right-anchored local rule $\bar\tau^{\,D}_{B,C}$. By the lemma,
Theorem~\ref{thm:inv-dynkin} is equivalent to saying that these two local rules agree on
invariants.

\begin{remark}\label{rem:inv-three-swaps}
There is a third natural way to interchange $B$ and $C$, namely the adjacent commutor
\[
q_{[i,i+1]}=1_A\otimes\sigma_{B,C}\otimes1_D.
\]
In general this is different from both $t_i$ and $t'_i$. Indeed, the equality
$t_i=q_{[i,i+1]}$ would amount to
\[
\sigma_{B,A\otimes C}
=
(1_A\otimes\sigma_{B,C})\circ(\sigma_{B,A}\otimes1_C),
\]
which is the hexagon axiom, and this fails for crystals.

Thus there are in general three different maps interchanging the same two adjacent factors:
$t_i$, $t'_i$, and $q_{[i,i+1]}$. On tableaux, $q_{[i,i+1]}$ is the crystal reflection operator;
it satisfies the braid relations, while the Bender--Knuth involutions do not. 
\end{remark}

\subsection{The atomic case}\label{sec:inv-atomic}

We first prove Theorem~\ref{thm:inv-dynkin} in the case where every factor is the single-box crystal
\[
B_1=\cdots=B_N=B(\varpi_1).
\]
This is the only place in the proof where tableaux enter.
An element of $B(\varpi_1)^{\otimes N}$ is a word
\[
u=u_1\cdots u_N,\qquad u_i\in[m].
\]
Under RSK, the crystal operators preserve the recording tableau $Q(u)$ and act on the insertion
tableau $P(u)$ \cite{Fulton}. Thus the connected components of the crystal are the fibres of
$Q(u)$. If $Q(u)$
has shape $\lambda$, the corresponding component is isomorphic to $B(\lambda)$. Thus, if $u$ is
invariant, its component must have
$
\lambda=(d^m)
$
for some $d$, and $N=dm$. We therefore obtain a bijection
\begin{equation}\label{eq:inv-syt}
\operatorname{Inv}\bigl(B(\varpi_1)^{\otimes N}\bigr)
\longleftrightarrow
\mathrm{SYT}(d^m).
\end{equation}
Thus the positions of the tensor factors become the entries $1,\dots,N$ of a rectangular standard
tableau, exactly as in Sections~\ref{sec:first}--\ref{sec:promotion}.

\begin{lemma}[the atomic dictionary]\label{lem:inv-dictionary}
Under the bijection \eqref{eq:inv-syt}, the interval reversal $q_{[1,r]}$ acts by evacuation on
the subtableau formed by the entries $1,\dots,r$, therefore by the Bender--Knuth word $q_{r-1}$.
Thus, the element $t_i\in\Cn_N$ acts by the Bender--Knuth involution $t_i$.
\end{lemma}

\begin{example}\label{ex:inv-dictionary}
Let $m=3$, $N=6$, and take the invariant word $u=332121$, with
\[
  Q(u)=\begin{ytableau}1&2\\3&5\\4&6\end{ytableau}.
\]

\emph{$r=3$.} On the first three factors,
\[
\Theta(332)=\mathrm{rc}(332)=211.
\]
The word $211$ is highest weight in its component, which is isomorphic to $B((2,1))$. Applying
lowering operators,
\[
211\xrightarrow{f_1}212\xrightarrow{f_2}213
\xrightarrow{f_2}313\xrightarrow{f_1}323,
\]
so $\xi(211)=323$. Then
\[
q_{[1,3]}(u)=323121,
\qquad
Q(323121)=\begin{ytableau}1&3\\2&5\\4&6\end{ytableau}.
\]
Thus the subtableau
$\begin{ytableau}1&2\\3\end{ytableau}$
is evacuated to
$\begin{ytableau}1&3\\2\end{ytableau}$,
while the entries $4,5,6$ stay in place.

\emph{$r=6$.} Since $u$ is invariant, Proposition~\ref{prop:inv-reversal} gives
\[
q_{[1,6]}(u)=\mathrm{rc}(u)=323211.
\]
Its recording tableau is
\[
  Q(323211)=\begin{ytableau}1&3\\2&4\\5&6\end{ytableau},
\]
which is the evacuation of $Q(u)$: on a rectangle, this is rotation by $180^\circ$ together with
$a\mapsto7-a$, as in Lemma~\ref{lem:theta-is-evac}.
\end{example}

\begin{proof}
Let
$
v=u_1\cdots u_r.
$
On the first $r$ factors, $q_{[1,r]}$ acts as $\xi\circ\Theta$, where
$
\Theta(v)=\mathrm{rc}(v)
$
is reverse-complement, with $a\mapsto m+1-a$.

A standard symmetry of RSK gives
\[
P(\mathrm{rc}(v))=\ev(P(v)),
\qquad
Q(\mathrm{rc}(v))=\ev(Q(v))
\]
\cite[App.~A1]{Fulton,Stanley}. The Sch\"utzenberger involution $\xi$ preserves the recording
tableau and acts by evacuation on the insertion tableau. Then
\[
P(q_{[1,r]}v)=P(v),
\qquad
Q(q_{[1,r]}v)=\ev(Q(v)).
\]
Since the insertion tableau after the first $r$ letters is unchanged, inserting the remaining
letters $u_{r+1},\dots,u_N$ adds the same boxes as before. Thus $q_{[1,r]}$ evacuates the
subtableau of $Q(u)$ formed by the entries $1,\dots,r$ and leaves the remaining entries in place.
By Lemma~\ref{lem:q-is-evac}, this is $q_{r-1}$.

Using
$
t_i
=
q_{[1,i]}q_{[1,i+1]}q_{[1,i]}q_{[1,i-1]}
$
in the cactus group, we conclude that its action on tableaux is
$
q_{i-1}q_iq_{i-1}q_{i-2}.
$
Since $q_i=q_{i-1}p_i$ and each $q_j$ is an involution, this is
$
p_ip_{i-1}^{-1},
$
which is precisely the Bender--Knuth involution $t_i$ on tableaux.
\end{proof}

\begin{proposition}[the atomic relation]\label{prop:inv-atomic}
On $\operatorname{Inv}(B(\varpi_1)^{\otimes N})$,
\[
t_j=t'_j,\qquad 1\le j<N.
\]
\end{proposition}

\begin{proof}
By Lemma~\ref{lem:inv-local},
\[
t'_{N-j}=q_{[1,N]}t_jq_{[1,N]}.
\]
Under the dictionary of Lemma~\ref{lem:inv-dictionary}, the right-hand side acts on
\eqref{eq:inv-syt} as
$
q\,t_j\,q.
$
By Theorem~\ref{thm:first-kind}, this is $t_{N-j}$. Replacing $N-j$ by $j$ gives the result.
\end{proof}

Only the standard-tableau case of Theorem~\ref{thm:first-kind} is used here. In this case $t_j$
interchanges the entries $j$ and $j+1$ unless they lie in the same row or column. Thus
Lemma~\ref{lem:theta-conj} reduces to the simple observation that rotation by $180^\circ$,
together with $a\mapsto N+1-a$, preserves this condition.

\subsection{Factorising a block interchange}\label{sec:inv-factor}

The local rules factor when one of the two moving blocks is itself a tensor product. This is
essentially Westbury's factorisation \cite[Lemmas~2, 3]{Westbury18}.

\begin{lemma}[factorisation]\label{lem:inv-factor}
For any crystals $A,B,C,D$,
\begin{align}
\tau^A_{B,\,C\otimes D}
&=
\tau^{A\otimes C}_{B,D}
\circ
\bigl(\tau^A_{B,C}\otimes1_D\bigr),
\label{eq:inv-fact1}\\
\tau^A_{B\otimes C,\,D}
&=
\bigl(\tau^A_{B,D}\otimes1_C\bigr)
\circ
\tau^{A\otimes B}_{C,D}.
\label{eq:inv-fact2}
\end{align}
The mirror identities hold for $\bar\tau$, with prefixes replaced by suffixes.
\end{lemma}

Thus \eqref{eq:inv-fact1} moves $B$ past $C\otimes D$ in two steps: first past $C$, anchored on
$A$, and then past $D$, now anchored on $A\otimes C$. Similarly, \eqref{eq:inv-fact2} moves $D$
left past $B\otimes C$ by first crossing $C$ and then $B$. 

\begin{proof}
Expanding the right-hand side of \eqref{eq:inv-fact1} gives
\[
\sigma_{B,A\otimes C\otimes D}
\circ
(\sigma_{A\otimes C,B}\otimes1_D)
\circ
(\sigma_{B,A\otimes C}\otimes1_D)
\circ
(\sigma_{A,B}\otimes1_C\otimes1_D).
\]
The two middle commutors cancel by Proposition~\ref{prop:coboundary}(a), leaving
$\tau^A_{B,C\otimes D}$.

For \eqref{eq:inv-fact2}, note that
\[
\tau^A_{C,B}\circ\tau^A_{B,C}=\id,
\]
since both maps are given by the same cactus generator $t_i$, and $t_i^2=1$. Apply
\eqref{eq:inv-fact1} with $(B,C,D)$ replaced by $(D,B,C)$ and invert the resulting identity.

The identities for $\bar\tau$ follow by conjugating with the full reversal. 
\end{proof}

Now suppose
\[
B=X_1\otimes\cdots\otimes X_b,
\qquad
C=Y_1\otimes\cdots\otimes Y_c,
\]
where the $X_p$ and $Y_r$ are atomic factors, and suppose that $A$ contains $u$ atomic factors.
Thus $X_p$ initially occupies position $u+p$, while $Y_r$ occupies position $u+b+r$.

Apply \eqref{eq:inv-fact2} first, splitting off $X_b$ from the left block:
\[
\tau^A_{B,C}
=
\bigl(\tau^A_{X_1\otimes\cdots\otimes X_{b-1},C}\otimes1_{X_b}\bigr)
\circ
\tau^{A\otimes X_1\otimes\cdots\otimes X_{b-1}}_{X_b,C}.
\]
Thus $X_b$ crosses the whole block $C$ first, then $X_{b-1}$ does the same, and so on.
Applying \eqref{eq:inv-fact1} to each of these block crossings, $X_p$ crosses
$Y_1,Y_2,\dots,Y_c$ in that order.

Each map in this decomposition has the form
$\tau^P_{X_p,Y_r}$,
where $P$ is the tensor product of the factors currently to the left of $X_p,Y_r$; in particular, the anchor changes after each crossing.

By Lemma~\ref{lem:inv-local}, if $X_p$ and $Y_r$ occupy adjacent positions $j,j+1$, the map
$\tau^P_{X_p,Y_r}$ is precisely the action of the cactus generator $t_j$. Applying the mirror identities in the opposite order, first the mirror of \eqref{eq:inv-fact1} to split off $X_b$ from $B$ and then the mirror of \eqref{eq:inv-fact2} to split $C$, gives the same sequence of crossings for $\bar\tau^{\,D}_{B,C}$: $X_b$ crosses $Y_1,\dots,Y_c$, then $X_{b-1}$ does the same, and so on, now with each $t_j$ replaced by $t'_j$. We therefore obtain the following.

\begin{corollary}\label{cor:inv-LR}
The maps $\tau^A_{B,C}$ and $\bar\tau^{\,D}_{B,C}$ factor into the same sequence of $bc$
adjacent box interchanges. In chronological order, their indices are
\begin{equation}\label{eq:inv-R}
 u+b,\dots,u+b+c-1;
 u+b-1,\dots,u+b+c-2;
 \dots;
 u+1,\dots,u+c.
\end{equation}
For $\tau^A_{B,C}$ the corresponding factors are $t_j$, while for
$\bar\tau^{\,D}_{B,C}$ they are $t'_j$.
\end{corollary}

For example, when $b=c=2$ and $u=0$, both factorizations have chronological index word
$(2,3,1,2)$.

\subsection{Propagation from the atomic case}\label{sec:inv-proof}

\begin{lemma}[atomic propagation]\label{lem:inv-propagation}
Suppose that
$
t_j=t'_j
$
on $\operatorname{Inv}(B(\varpi_1)^{\otimes N})$ for every $N$ and every $j$. Then
$
t_i=t'_i
$
on
\[
\operatorname{Inv}\bigl(B(\lambda_1)\otimes\cdots\otimes B(\lambda_n)\bigr)
\]
for all $\lambda_1,\dots,\lambda_n$ and all $i$.
\end{lemma}

\begin{proof}
Fix $i$, and write $A,B,C,D$ as in Section~\ref{sec:inv-local}, with
\[
B=B(\lambda_i),\qquad C=B(\lambda_{i+1}).
\]

\emph{Refine into boxes.}
For each $j$, use the row-reading embedding
\[
\iota_j\colon
B(\lambda_j)\hookrightarrow B(\varpi_1)^{\otimes|\lambda_j|},
\]
where the rows of a tableau are read from left to right, starting with the bottom row and moving
upward. Its image is a connected component of $B(\varpi_1)^{\otimes|\lambda_j|}$. Tensoring these inclusions gives
\[
\iota\colon
B(\lambda_1)\otimes\cdots\otimes B(\lambda_n)
\hookrightarrow
B(\varpi_1)^{\otimes N},
\qquad
N=\sum_j|\lambda_j|.
\]
The image is a union of connected components, so invariants map to invariants.

Write $\widetilde A,\widetilde B,\widetilde C,\widetilde D$ for the tensor products of single-box
crystals obtained by refining the factors in $A,B,C,D$. Thus
\[
\widetilde B=B(\varpi_1)^{\otimes b},
\qquad
\widetilde C=B(\varpi_1)^{\otimes c},
\qquad
b=|\lambda_i|,
\qquad
c=|\lambda_{i+1}|.
\]
Let
\[
\iota_{\mathrm{src}}\colon
A\otimes B\otimes C\otimes D
\hookrightarrow
\widetilde A\otimes\widetilde B\otimes\widetilde C\otimes\widetilde D
\]
be the tensor product of the row-reading embeddings above, and let
\[
\iota_{\mathrm{tgt}}\colon
A\otimes C\otimes B\otimes D
\hookrightarrow
\widetilde A\otimes\widetilde C\otimes\widetilde B\otimes\widetilde D
\]
be the corresponding embedding after the two middle factors have been interchanged.

Set
\[
L:=\tau^{\widetilde A}_{\widetilde B,\widetilde C},
\qquad
R:=\bar\tau^{\,\widetilde D}_{\widetilde B,\widetilde C}.
\]
Thus $L$ and $R$ interchange the two blocks of $b$ and $c$ single-box factors corresponding to
$B$ and $C$, anchored respectively on the refined blocks $\widetilde A$ and $\widetilde D$.
Since the $\iota_j$ are inclusions of connected components, Lemma~\ref{lem:inv-natural} gives
\[
\iota_{\mathrm{tgt}}\circ\tau^A_{B,C}
=
L\circ\iota_{\mathrm{src}},
\qquad
\iota_{\mathrm{tgt}}\circ\bar\tau^{\,D}_{B,C}
=
R\circ\iota_{\mathrm{src}}.
\]
Thus the local rules on $B$ and $C$ are the restrictions of the corresponding block interchanges
after the factors have been refined into single boxes.

\emph{Common factorization.}
By Corollary~\ref{cor:inv-LR}, $L$ and $R$ factor into the same chronological sequence of
adjacent positions, given by \eqref{eq:inv-R}. At each step the factor for $L$ is $t_j$, while
the corresponding factor for $R$ is $t'_j$. Starting from an invariant, the two intermediate
tensors agree inductively: every local rule is a crystal isomorphism and hence preserves
invariants, and the hypothesis gives $t_j=t'_j$ on invariants at the next crossing. Therefore
$L=R$ on $\operatorname{Inv}(B(\varpi_1)^{\otimes N})$.

\emph{Pull back.}
If
\[
x\in
\operatorname{Inv}\bigl(B(\lambda_1)\otimes\cdots\otimes B(\lambda_n)\bigr),
\]
then $\iota_{\mathrm{src}}(x)$ is an invariant of $B(\varpi_1)^{\otimes N}$. By the naturality
identities above,
\[
\iota_{\mathrm{tgt}}\bigl(\tau^A_{B,C}x\bigr)
=
L\bigl(\iota_{\mathrm{src}}(x)\bigr),
\qquad
\iota_{\mathrm{tgt}}\bigl(\bar\tau^{\,D}_{B,C}x\bigr)
=
R\bigl(\iota_{\mathrm{src}}(x)\bigr).
\]
We have just proved that $L=R$ on invariants, so
$
L\bigl(\iota_{\mathrm{src}}(x)\bigr)
=
R\bigl(\iota_{\mathrm{src}}(x)\bigr).
$
Combining these equalities gives
$
\iota_{\mathrm{tgt}}\bigl(\tau^A_{B,C}x\bigr)
=
\iota_{\mathrm{tgt}}\bigl(\bar\tau^{\,D}_{B,C}x\bigr).
$
Since $\iota_{\mathrm{tgt}}$ is injective,
\[
\tau^A_{B,C}x=\bar\tau^{\,D}_{B,C}x.
\]
By Lemma~\ref{lem:inv-local}, this is exactly $t_ix=t'_ix$.
\end{proof}

\begin{proof}[Proof of Theorem~\ref{thm:inv-dynkin}]
A finite normal crystal is a disjoint union of irreducible crystals, and the cactus action
respects this decomposition. We may therefore assume that every $B_j$ is irreducible.
Proposition~\ref{prop:inv-atomic} gives the hypothesis of
Lemma~\ref{lem:inv-propagation}, so
$
t_j=t'_j
$
on invariants for every $j$. Lemma~\ref{lem:inv-local} now gives
\[
q_{[1,n]}t_iq_{[1,n]}
=
t'_{n-i}
=
t_{n-i},
\]
which is the desired Dynkin-reversal relation.
\end{proof}

\begin{proof}[Proof of Theorem~\ref{thm:inv-adapted}]
By Theorem~\ref{thm:inv-dynkin}, the cactus action on invariants satisfies the Dynkin-reversal
relations, and thus factors through $\RC_n$. By Theorem~\ref{thm:equivalence},
\[
Q_D=q_{[1,n]}
\]
in $\RC_n$ for every Dynkin path $D$. By Lemma~\ref{lem:reduced-adapted}, the words in
$\mathcal K(D)$ are exactly the reduced words for $w_0$ adapted to the orientation $\Omega_D$,
and by Lemma~\ref{lem:knitting-words} every such word represents $Q_D$. Thus every adapted
reduced word acts on invariants as $q_{[1,n]}$. Proposition~\ref{prop:inv-reversal} identifies
this action with
\[
b_1\otimes\cdots\otimes b_n
\longmapsto
\xi(b_n)\otimes\cdots\otimes\xi(b_1),
\]
as claimed.
\end{proof}

\subsection{A topological proof and extension to other types}
\label{subsec:inv-topological}

Theorem~\ref{thm:big-cactus} gives a second proof of
Theorem~\ref{thm:inv-dynkin}, valid in every semisimple type.

Let $\mathfrak g$ be a complex semisimple Lie algebra and
$\boldsymbol\lambda=(\lambda_1,\ldots,\lambda_n)$ a tuple of dominant
integral weights. Set
\[
 W_{\boldsymbol\lambda}
   =\bigl(V(\lambda_1)\otimes\cdots\otimes V(\lambda_n)\bigr)^{\mathfrak g}.
\]
For $z\in X_n=\overline M_{0,n+1}(\mathbb R)$, the rational Gaudin
algebra $\mathcal A_z$ acts on $W_{\boldsymbol\lambda}$ with simple
spectrum. Its joint eigenlines form an unramified covering
$\mathcal E^0_{\boldsymbol\lambda}\to X_n$.
The Gaudin monodromy theorem \cite[Theorems~1.4 and~8.5]{HKRW}
identifies its monodromy with the cactus action on
$\operatorname{Inv}(B(\lambda_1)\otimes\cdots\otimes B(\lambda_n))$.
The identification is compatible with permuting the marked points and
the labeled factors. For unequal factors we consider all orderings,
as in Section~\ref{sec:inv-action}.

\begin{lemma}\label{lem:inv-gaudin-descent}
Forgetting the root $f:X_n\to Y_n=\overline M_{0,n}(\mathbb R)$ gives
\[
 \mathcal E^0_{\boldsymbol\lambda}
   \cong f^*\mathcal F_{\boldsymbol\lambda}
\]
for an unramified covering $\mathcal F_{\boldsymbol\lambda}\to Y_n$.
This descent is compatible with simultaneous permutations of the
markings and the labeled factors.
\end{lemma}

\begin{proof}
Put $\mathcal U=U(\mathfrak g)^{\otimes n}$. Since the diagonal
copy $\Delta(\mathfrak g)$ annihilates $W_{\boldsymbol\lambda}$,
the action of $\mathcal U^{\mathfrak g}$ on this space factors through
the quantum Hamiltonian reduction
\[
 \bigl(\mathcal U/\mathcal U\Delta(\mathfrak g)\bigr)^{\mathfrak g},
\]
where $\mathcal U\Delta(\mathfrak g)$ is a left ideal. By
\cite[Section~2.6 and Theorem~4.15(2)]{IKR}, the image of the full
homogeneous Gaudin algebra in this reduction is invariant under
simultaneous projective transformations of its marked points.
For semisimple $\mathfrak g$, apply this statement to each simple
summand. Hence, on a smooth curve, the image in
$\operatorname{End}(W_{\boldsymbol\lambda})$, and therefore its joint
eigenline decomposition, depends only on the $n$ original markings,
not on the additional root.

Here is the extension to the boundary. Using the section $\sigma_n$
from the proof of Theorem~\ref{thm:big-cactus}, define
$\mathcal F_{\boldsymbol\lambda}=\sigma_n^*\mathcal E^0_{\boldsymbol\lambda}$.
Let $\mathcal D(z)$ be the unordered set of joint eigenlines at $z$,
viewed in $\operatorname{Sym}^{\dim W_{\boldsymbol\lambda}}
\mathbb P(W_{\boldsymbol\lambda})$; the case $W_{\boldsymbol\lambda}=0$
is immediate. This set varies continuously and has distinct elements,
including at the boundary. For a smooth $z$, move the root towards the
marking $n$. Root-independence and continuity give
\[
 \mathcal D(z)=\mathcal D\bigl(\sigma_n(f(z))\bigr).
\]
Both sides are continuous, so density of the smooth real locus gives
the equality for every $z\in X_n$. Identifying the same eigenlines
proves the asserted pullback isomorphism. Its description is independent
of the chosen lift of a point of $Y_n$, so the permutation maps on the
original coverings descend as well.
\end{proof}

\begin{theorem}\label{thm:inv-all-types}
For any complex semisimple $\mathfrak g$ and finite normal
$\mathfrak g$-crystals $B_1,\ldots,B_n$, $n\geq3$, the cactus action
on their invariant tensors factors through
\[
 \Cn_n\xrightarrow{f_*}J'_n\cong\RC_n.
\]
In particular, Theorem~\ref{thm:inv-dynkin} holds in every semisimple
type.
\end{theorem}

\begin{proof}
For irreducible factors, apply the Gaudin monodromy theorem to
Lemma~\ref{lem:inv-gaudin-descent}. Monodromy of a pullback is obtained
by composing with the induced map on fundamental groups. Thus the
action factors through $f_*$, whose kernel is given by
Theorem~\ref{thm:big-cactus}. Naturality of the commutor for inclusions
of connected components reduces arbitrary finite normal crystals to
this case.
\end{proof}

\begin{corollary}\label{cor:inv-adapted-all-types}
Theorem~\ref{thm:inv-adapted} holds for finite normal crystals of any
complex semisimple Lie algebra: every adapted reduced word acts on
invariant tensors as
\[
 b_1\otimes\cdots\otimes b_n
 \longmapsto
 \xi_{B_n}(b_n)\otimes\cdots\otimes\xi_{B_1}(b_1).
\]
\end{corollary}

\begin{proof}
Use Theorems~\ref{thm:inv-all-types} and~\ref{thm:equivalence}.
The proof of Proposition~\ref{prop:inv-reversal} uses only the
commutor formula and the fact that the Sch\"utzenberger involution
fixes singleton components, so it applies in every semisimple type.
\end{proof}

The diagram $A_{n-1}$ in this corollary indexes tensor positions, not
the simple roots of $\mathfrak g$. The invariant condition is essential:
on a general multiplicity space
$\operatorname{Hom}_{\mathfrak g}(V(\mu),\bigotimes_iV(\lambda_i))$,
the root carries $V(\mu)^*$ rather than the trivial representation,
so the preceding root-forgetting argument does not apply. The assertions for
$n\leq2$ follow directly from the commutor formula.
\subsection{Rectangular tableaux as invariants}\label{sec:inv-rows}

Finally, we explain how the rectangular tableaux of Sections~\ref{sec:first}--\ref{sec:promotion} arise from Theorem~\ref{thm:inv-adapted} when the factors are single rows.

Let $R_a=B(a\varpi_1)$ be the crystal of a single row of length $a$. Under the row-reading
embedding, an element of $R_a$ is simply a weakly increasing word of length $a$. Thus
\[
r_1\otimes\cdots\otimes r_n
\in
R_{c_1}\otimes\cdots\otimes R_{c_n}
\]
refines to the word
\[
u=r_1\cdots r_n
\in B(\varpi_1)^{\otimes N},
\qquad
N=c_1+\cdots+c_n.
\]
Set
\[
C_j=c_1+\cdots+c_j,\qquad C_0=0.
\]

Under RSK, the connected component of $u$ is determined by its recording tableau $Q(u)$. Since
$r_j$ is weakly increasing, the entries
\[
C_{j-1}+1,\dots,C_j
\]
form a horizontal strip in $Q(u)$, with the entries increasing from left to right. Conversely,
this condition is exactly what is needed for the corresponding block of the word to be weakly
increasing. Replacing all entries in this interval by $j$, for each $j$, therefore gives a
semistandard tableau $T(u)$ of the same shape and content $(c_1,\dots,c_n)$, and standardisation
recovers $Q(u)$.

The word $u$ is invariant exactly when its RSK component is a one-element
$\mathfrak{sl}_m$-crystal, or equivalently when $Q(u)$ has rectangular shape $(k^m)$. Therefore
\begin{equation}\label{eq:inv-rect}
  \operatorname{Inv}\big(R_{c_1}\otimes\cdots\otimes R_{c_n}\big)\;\longleftrightarrow\;
  \SSYT\big((k^m),n\big)_{\text{content }(c_1,\dots,c_n)}.
\end{equation}
Thus the $j$-th row factor corresponds to the letter $j$ in the tableau.

The same description identifies the cactus action. First note that the full reversal of a single
row factor acts trivially on its row-reading image:
\[
\rho(r)=\xi(\mathrm{rc}(r))=r.
\]
Consequently, after refining the first $k$ row factors into boxes, the coarse reversal
$q_{[1,k]}$ is the restriction of the fine reversal $q_{[1,C_k]}$ on the first $C_k$ boxes:
by \eqref{eq:inv-sigma-positions}, the only additional maps are the internal reversals of the row
blocks, and these act trivially.

By Lemma~\ref{lem:inv-dictionary}, the fine reversal $q_{[1,C_k]}$ evacuates the subtableau of
$Q(u)$ formed by the entries $1,\dots,C_k$. Evacuation commutes with standardisation, by the
compatibility of RSK with standardisation and reverse-complement. Hence after destandardisation
this becomes evacuation of the letters $1,\dots,k$ in $T(u)$. In the notation of the first part
of the paper, this is precisely $q_{k-1}$. Thus, under \eqref{eq:inv-rect},
\[
q_{[1,k]}\longmapsto q_{k-1}.
\]

In particular, $q_{[1,n]}$ acts as $q_{n-1}=q$, which on a rectangular tableau is $\theta$ by
Lemma~\ref{lem:theta-is-evac}. Thus Proposition~\ref{prop:inv-reversal} recovers that lemma.
Similarly,
\[
t_j
=
q_{[1,j]}q_{[1,j+1]}q_{[1,j]}q_{[1,j-1]}
\]
acts as
\[
q_{j-1}q_jq_{j-1}q_{j-2}=t_j,
\]
the Bender--Knuth involution on $T(u)$. It follows that $\BKR_n$ is the image of $\Cn_n$ on the
invariants of tensor products of rows, and Theorem~\ref{thm:knit-are-evac} is precisely
Theorem~\ref{thm:inv-adapted} for row factors.

\section{Related work}\label{sec:related}

\subsection{Bender--Knuth groups}

The group $\BK_n$ was introduced by Berenstein and Kirillov~\cite{BK}, who found, among other
relations, $(t_1t_2)^6=1$. Its connection with the cactus group was established independently
in~\cite{CGP}, \cite{Halacheva} and~\cite{BK16}.

Several related Bender--Knuth and partial-evacuation actions have since been studied. Rodrigues
\cite{Rodrigues} defines a shifted Berenstein--Kirillov group and shows that it is a quotient of
the cactus group. Azenhas, Tarighat Feller and Torres~\cite{ATF} define a symplectic
Berenstein--Kirillov group acting on Kashiwara--Nakashima tableaux and exhibit two relations
which do not follow from the defining relations of the symplectic cactus group. Liao and
the third author~\cite{LR} study the image of the cactus group on standard Young tableaux of a fixed
shape. Bender--Knuth moves on linear extensions of posets are studied in~\cite{REU}, where the
cactus relations are shown to hold for all minuscule posets.

The rectangular group $\BKR_n$ fits naturally into this picture. Conjecture~\ref{conj:relative}
says that the additional relations obtained by restricting from arbitrary tableaux to rectangles
are exactly the Dynkin-reversal relations.

\subsection{The big cactus group}
The big cactus group $J'n=\pi_1^{\Sym_n}(\overline M{0,n}(\mathbb R))$ was introduced by Henriques~\cite{HenriquesBig}. The space $\overline M_{0,n}(\mathbb R)$ is tiled by associahedra, with one tile for each dihedral order of the markings, and is aspherical~\cite{DJS}. Davis, Januszkiewicz and Scott studied the symmetry group $\operatorname{Aut}(X)$ of the corresponding associahedral tiling of the universal cover, which in our notation identifies with $J'_n$, and gave a presentation with generators indexed by subpolygons of an $n$-gon and by its diagonals~\cite[Thm.~8.6.2 and Rem.~8.6.5]{DJS}. Theorem~\ref{thm:big-cactus} instead presents $J'_n$ as the quotient of the ordinary cactus group by the $n-1$ Dynkin-reversal relations, and Corollary~\ref{cor:big-two-presentations} gives a second presentation by adapted words. The proof uses the mirabolic cactus group and the caterpillar realization of the elements $t_i$ from forthcoming work of Gurenkova, Halacheva and the third author~\cite{GHR}.

\subsection{Quivers and toggles}

In the language of Garver, Patrias and Thomas~\cite{GPT}, a flip of a Dynkin path at a source
corresponds to reflection at a source of the associated quiver. The knitting region
$R(D,\bar D)$ can be identified with the Auslander--Reiten quiver of
$\mathrm{rep}(\Omega_D)$; see Remark~\ref{rem:sinks-sources}. Garver, Patrias and Thomas show
that source reflections are realized by piecewise-linear toggles on minuscule reverse plane
partitions. Our knitting argument similarly follows source reflections along a full adapted reduced word for $w_0$; on rectangles the resulting action is evacuation.

Dranowski, Elek, Kamnitzer and Morton-Ferguson~\cite{DEKM} develop a uniform toggle model for
minuscule heaps in simply-laced type. Brown, Elek and Halacheva~\cite{BEH} relate this toggle
action to the cactus action for the crystals $B(n\varpi_1)$ in type $D$. This suggests asking for
an analogue of Theorem~\ref{thm:equivalence} in other types.

\subsection{Sorting networks}

Warrington~\cite{Warrington} introduced the four reentrant subnetworks $X$ as a combinatorial
analogue of Sylvester's four-point problem, using the Goodman--Pollack correspondence with planar
point configurations~\cite{GoodmanPollack}. He conjectured that a random $4$-subnetwork of a
random sorting network is reentrant with probability $1/4$; this was proved by Angel and
Holroyd~\cite{AH}. Meng~\cite{Meng} studied $X$-avoiding reduced words and asked for their
number.

Theorem~\ref{thm:X-avoiding} gives one direction of a surprisingly close connection:
every adapted reduced word is $X$-avoiding. Conjecture~\ref{conj:X} says that the converse
also holds. We verified the conjecture directly for $n\le6$, and for $n=7$ by comparison of counts. 

If the conjecture holds in general, Meng's question becomes the enumeration of adapted reduced
words, equivalently of linear extensions of the Auslander--Reiten quivers associated to the
orientations of $A_{n-1}$. We have no conceptual explanation for why adaptedness should be
governed by a convexity condition on four points.

\subsection{Promotion}

Rectangular promotion periodicity and some of its proofs are discussed in
Remark~\ref{rem:other-proofs}. On tableaux, the identity of
Theorem~\ref{thm:promotion-identity} is the Bender--Knuth form of the classical relation between
promotion, evacuation and dual evacuation; see Stanley's survey~\cite{Stanley09}. 
\subsection{Crystals and invariant tensors}

The cactus action on tensor products of crystals was constructed by
Henriques and Kamnitzer~\cite{HK}. In joint work with Halacheva,
Kamnitzer and Weekes, the third author~\cite{HKRW} proved Etingof's
conjecture identifying this action with the monodromy of Bethe
eigenvectors for Gaudin algebras on tensor-product multiplicity spaces.
The invariants considered in Section~\ref{sec:invariants} correspond
to the multiplicity space of the trivial representation.

Halacheva~\cite{Halacheva} used skew Howe duality to identify the internal cactus action
on tableaux with the external action on tensor products of column crystals. For row
crystals, Section~\ref{sec:inv-rows} gives the analogous identification directly by RSK
and standardisation. Neither identification is used in the proof of
Theorem~\ref{thm:inv-dynkin}.

Westbury~\cite{Westbury18} introduced local rules in coboundary categories; the factorisation in
Lemma~\ref{lem:inv-factor} is essentially his factorisation of these local rules. Pfannerer,
Rubey and Westbury~\cite{PRW} use local rules and growth diagrams to study promotion and
evacuation on highest weight words, relating them to rotation and reversal of chord diagrams
and permutations.

There are also several general periodicity results for promotion on invariants. For tensor powers
of a single crystal, in arbitrary semisimple type, Westbury~\cite{Westbury16} obtains promotion
periodicity as part of a cyclic sieving phenomenon. Fontaine and Kamnitzer~\cite{FK} prove a
corresponding result for rotation on invariant vectors in tensor products of minuscule
representations.

Theorems~\ref{thm:promotion-identity} and~\ref{thm:inv-all-types} also give promotion periodicity for arbitrary finite normal crystals $B_1,\dots,B_n$: the $n$-fold composite of the promotion maps through the cyclic orderings of the factors is the identity on $\operatorname{Inv}(B_1\otimes\cdots\otimes B_n)$. 

Finally, Borodin~\cite{Borodin} studies the cactus action on $\mathfrak{sl}_2$ multiplicity sets
through arc diagrams and finds additional relations, including braid relations.

\subsection{Affine cactus groups and Kirillov--Reshetikhin crystals}
\label{subsec:related-affine-cactus}

For $1\leq r<n$, the tableaux $\SSYT((s^r),n)$ form the underlying
finite-type crystal of the Kirillov--Reshetikhin crystal $B^{r,s}$
of type $A_{n-1}^{(1)}$. Under the reality and separation hypotheses
of their construction, Krylov, Mashanova-Golikova and the third author
\cite[Theorems~17.4 and~18.2]{KMR} identify Yangian Bethe eigenlines
with tensor products of Kirillov--Reshetikhin crystals and identify
the interval monodromies with partial Sch\"utzenberger involutions
for finite-type subdiagrams of the affine Dynkin diagram.
This extends the internal cactus action to the affine setting.
A complementary monodromy description for the trigonometric Gaudin
model, in terms of crystal commutors and cyclic permutation of tensor
factors, is given in \cite[Theorem~1.3(3)]{KRFlowers}; this is an
external action on tensor positions, rather than an internal action
on Dynkin colours.

The affine parameter space is the compact real form
$Z_n=\overline M_{0,n+2}^{\mathrm{comp}}(\mathbb R)$, on which
$1,\ldots,n$ are real markings and $0,\infty$ are exchanged by the
real structure. Its $\Sym_n$-equivariant fundamental group is the
extended affine cactus group $\widetilde{AC}_n$
\cite[Theorem~11.12]{IKLPR}. There is a natural comparison with the
spaces used here. Forgetting $0,\infty$ defines $F:Z_n\to Y_n$,
while attaching a rigid three-pointed component carrying this
conjugate pair at the root defines $j:X_n\to Z_n$. Since the attached
component is contracted after the pair is forgotten, $F\circ j=f$.
Thus, with compatible basepoints, we obtain
\[
 \Cn_n\xrightarrow{j_*}\widetilde{AC}_n
 \xrightarrow{F_*}J'_n,
 \qquad F_*j_*=f_*.
\]
This compares the groups and parameter spaces, but does not imply that the full Kirillov--Reshetikhin monodromy descends along $F$.

\medskip

{\footnotesize
\noindent\textbf{Gabriel Frieden}\par
\noindent Desautels Faculty of Management, McGill University, Montreal QC, Canada\par
\noindent\texttt{gabriel.frieden@mcgill.ca}\par
}

\medskip

{\footnotesize
\noindent\textbf{Pavlo Pylyavskyy}\par
\noindent School of Mathematics, University of Minnesota, Minneapolis, Minnesota, United States\par
\noindent\texttt{ppylyavs@umn.edu}\par
\noindent P.P. was supported by a Simons Foundation International grant administered by the Simons Foundation, SFI-MPS-TSM-00013120.\par
}

\medskip
{\footnotesize
\noindent\textbf{Leonid Rybnikov}\par
\noindent Department of Mathematics and Statistics, University of Montreal, Montreal QC, Canada\par
\noindent\texttt{leonid.rybnikov@umontreal.ca}\par
\noindent L.R. was supported by Fondation Courtois and by the NSERC Discovery program.\par
}

\end{document}